\documentclass[11pt,
]{amsart}

\usepackage{
amssymb, graphicx
}

\usepackage{
xcolor,
cite
}
 
\date{\today}
\usepackage{datetime}

\usepackage{enumerate}
\numberwithin{equation}{section}

\newtheorem{theorem}{Theorem}[section]
\newtheorem{proposition}{Proposition}[section]

\theoremstyle{definition}
\newtheorem{remark}{Remark}[section]

\DeclareMathOperator{\supp}{supp}

\DeclareMathOperator{\sinc}{sinc}

\DeclareMathOperator{\sign}{sign}
\newcommand{\eps}{\varepsilon}

\newcommand{\g}{ f_0}
\newcommand{\R}{\mathbb{R}}

\renewcommand{\r}[1]{{\eqref{#1}}}

\newcommand{\be}[1]{\begin{equation}\label{#1}}
\newcommand{\ee}{\end{equation}}

\newcommand{\p}{{\partial}}

\renewcommand{\d}{\mathrm{d}}

\renewcommand{\i}{\mathrm{i}}
\newcommand{\bo}{\partial \Omega}

\newtheorem{example}{Example}[section]

\title[Recovery of a general nonlinearity in the semilinear wave equation]{Recovery of a general  nonlinearity in the semilinear wave equation}

\author[A. S\'a Barreto]{Ant\^ onio S\'a Barreto}
\author[P. Stefanov]{Plamen Stefanov}
\address{Department of Mathematics, Purdue University, West Lafayette, IN 47907}
\thanks{The first author is partly supported by the Simons Foundation grants \#349507 and  \#848410. 
The second author is partly supported by the  NSF Grant DMS-1900475.}

\begin{document}
 \begin{abstract}
We study the inverse problem of recovery a  non-linearity $f(x,u)$, 
which is compactly supported in $x$, in the semilinear wave equation $u_{tt}-\Delta u+ f(x,u)=0$. We probe the medium with either complex or real-valued harmonic waves of wavelength $\sim h$ and amplitude $\sim 1$.  They propagate in a regime where the non-linearity affects the subprincipal but not the principal term, except for the zeroth harmonics. We measure the transmitted wave when it exits $\supp_x f$. We show that one can recover $f(x,u)$ when it is an odd function of $u$, and we can recover $\alpha(x)$ when $f(x,u)=\alpha(x)u^{2m}$. This is done in an explicit way as $h\to0$.
\end{abstract} 
\maketitle

\section{Introduction}  
Consider the semilinear wave equation
 \be{1}
u_{tt}-\Delta u+ 
f(x,u)=0,\quad (t,x)\in \R_t\times\R_x^n.
\ee
We assume that $f$ is smooth and compactly supported in the $x$ variable.

The problem we study is whether we can recover $f(x,u)$  
for all $x$ and $u$ from remote measurements. We show that this can be done in an explicit way when $f$ is an odd function of $u$, which covers the case 
\be{1odd}
f(x,u) = \g(x,|u|^2)u
\ee
 studied extensively in the literature. Without that assumption, we show that we can recover $\alpha(x)$ explicitly  when $f(x,u)=\alpha(x)|u|^m$, with $m$ (an even) integer. We probe the non-linearity with high-frequency incident waves of wavelength $\sim h$, $0<h\ll1$, and look at the asymptotic expansion of the wave at some exit time $t=T$ as $h\to0$. We take both complex and real incident waves. The real case is harder and perhaps more relevant for applications. While solvability with not necessarily small initial conditions is not guaranteed unless we make additional assumptions on $f$, we show that \r{1} is solvable for $t\in [0,T]$ with the waves we use the probe the medium.

In the works on inverse problems for non-linear hyperbolic PDEs so far, it is usually assumed that $u$ is small: one takes an asymptotic expansion of a solution with initial conditions $u_0=\eps_1 u_1+\eps_2 u_2 +\eps_3 u_3 +\eps_4 u_4$ chosen so that their weak singularities collide at a chosen point at a chosen time. Then one takes the limit as all $\epsilon$'s tend to zero. The information about the non-linearity is extracted from the $\eps_1\eps_2 \eps_3\eps_4$ term in the asymptotic expansion, which has a certain weak singularity; this method is sometimes called \textit{higher order linearization}. This can and does provide information about the Taylor expansion of $f(x,u)$ w.r.t.\ $u$ at $u=0$ only, see \cite{LassasUW_2016}. 
In \cite{OSSU-principal}, one has three waves instead of four. 
An exception is the work \cite{Antonio2020inverse}, where $\g$ is $x$-independent (and the nonlinearity $f(x,u)$ is critical, $\sim |u|^4u$, $n=3$); then $f$ is a similar perturbation of a fixed non-zero solution $f_0$ rather than a perturbation of the zero one. Recently in \cite{S-Antonio-nonlinear}, the authors studied a cubic non-linearity $\g(x,|u|^2)u=\alpha(x)|u|^2u$ and proposed using non-small solutions producing an exit signal of magnitude comparable to the one of the incident field. The solution then is $u\sim h^{-1/2}$, and  propagates in the weakly nonlinear regime. We showed that the propagating wave has a phase shift in the principal term which is proportional to the X-ray transform of $\alpha$ along the characteristic rays; then one can recover $\alpha$ from that information. The method in \cite{S-Antonio-nonlinear} can be extended to $f(x,u)$ having that type of cubic asymptotic behavior as $|u|\to\infty$. Aside from those two extremes, $u\to0$ or $u\to\infty$, the inverse problem remained open to author's best knowledge, see also section~\ref{sec_discussion}. 

As we emphasized above, methods based on asymptotically small solutions can only recover the non-linearity at $u=0$, and possibly its Taylor expansion there, see for example \cite{Hintz1,Hintz2,LassasUW_2016, LUW1, lassas2020uniqueness}. 
Moreover, they propagate essentially linearly, where the non-linearity is negligible. To recover the non-linearity away from $u=0$, we need non-small solutions; and we need a problem allowing such solutions. The main idea of this work is to use high frequency solutions $u\sim 1$ (in the $L^\infty$ norm). When $f$ is an odd function of $u$ as in \r{1odd}, that would put us in the principally linear geometric optics regime (see \cite{Metivier-Notes} and section~\ref{sec_discussion}) and the non-linearity would affect the sub-principal term.   The general case is more delicate and does affect the principal term. If the principal part is  $u\sim \chi(\phi)e^{\i\phi/h}$, with $\phi=-t+x\cdot\omega$, and $0\le \chi\in C_0^\infty$, then its modulus is $|\chi(\phi)|$, and in the PDE, we would have $\g(x,\chi^2)u$ modulo $O(h)$. So the non-linearity would act effectively as a time-dependent potential $V(x)=\g (x,\chi^2(\phi))$, which can be recovered from near-field scattering data by means of its X-ray transform. In other words, once we justify that claim, we get an essentially an  inverse problem for the linear wave equation $\Box u+V(t,x)u=0$, which is well studied. Note that this is not a linearization. The X-ray transform of $V$ is contained in the subprincipal term of the exit wave. Thus we can recover $\g(x,p)$ for every $p$ in the range of $\chi^2$.  

Assume we have real incident fields like $u\sim \chi(\phi) \cos(\phi/h)$; then the situation is quite different. This is a solution of the free (linear) wave equation, and would be the principal part of the non-linear solution when $f$ is of the type \r{1odd}. Its square is not $\chi^2(\phi)$ anymore. 
The ``effective potential'' $V=\g(x,\chi(\phi)^2\cos^2(\phi/h))$ would be highly oscillatory, and we need to expand it (multiplied by $u$) in Fourier modes first. This leads to expanding the function $\g(x,M^2q^2)q$ in Chebyshev polynomials over the interval $|q|\le 1$, see \r{R3} below. Then we show that the wave develops (odd) harmonics and each has an amplitude proportional to the X-ray transform of the corresponding Chebyshev coefficient. This allows us to recover those coefficients and ultimately, $\g$. We also show that the first harmonic only is enough: it leads to an Abel equation, see \r{R9}, allowing us to recover $\g$. 

The case of general $f$ not necessarily an odd function of $u$ is more delicate. We study real incident waves only. Then the self-interaction of the wave creates a  zeroth harmonic $u^{(0)}$, which is known in the physics literature as rectification.  Then  $u^{(0)}$ affects the principal symbol, in fact, and its principal part solves another semi-linear wave equation with a non-linearity depending on $f$, see \r{E12}. The zeroth harmonic  $u^{(0)}$ mixes (interacts) with the non-zero ones and affects all the frequencies of the exit signal. For this reason, it is not clear how to extract information about $f$ from them. When $f$ is a polynomial in $u$, with $x$ dependent coefficients, the highest harmonic does not depend on  $u^{(0)}$, and we can actually recover the principal term of that polynomial; in particular, we can recover $\alpha_m$ when $f=\alpha_m(x) u^m$. 

We refer to section~\ref{sec_discussion} for a further discussion of relevant works on inverse problems for semi-linear hyperbolic PDEs and non-linear geometric optics. 

\smallskip
\textbf{Setup and main results.} 
We describe our setup now. Assume $\supp_x\g\subset B(0,R)=\{x;\; |x|<R\}$. We take $n=2$ or $n=3$. 
We are probing the medium with
\be{10a}
u_\textrm{in}^\mathbb{C}= e^{\i (-t+x\cdot\omega)/h} \chi(-t+x\cdot\omega), \quad \omega\in S^{n-1}. 
\ee
Assume $C_0^\infty \ni \chi\ge0$ and let $\delta>0$ be such that  $\supp\chi\subset (-\delta, \delta)$. We solve \r{1} with initial condition
\be{1c}
u=u_\text{in}^\mathbb{C}\quad \text{for $t<-R-\delta$}
\ee
which solves the free wave equation for such $t$. We measure 
\be{2a}
\Lambda(u_\text{in}^\mathbb{C})(x,\omega,h) = u|_{t=T, \; |x\cdot\omega-T|\le\delta},
\ee
where $T>R+\delta$ is fixed and   $u$ is the unique bounded solution. 

Our first main theorem is the following. 

\begin{theorem}\label{thm_main} 
Let the complex  $u_\text{\rm in}^\mathbb{C}$ be as in \r{10a}, and let $K=\max\chi$. 
Assume that $f$ is of the form \r{1odd} for $|u|\le K$ with $\g$ smooth.  Then, for $h\ll1$,  there is a unique bounded solution to \r{1}, \r{1c}  defined for $t\in [0,T]$. Moreover,

(a)  
\be{main_eq}
\begin{split}
\Lambda(u_\text{in}^\mathbb{C})  &= e^{\i (-T+x\cdot\omega)/h} \chi(-T+x\cdot\omega) \\
&\quad \times \left(  
1+ \i \frac{h}2\int  \g\big(x+s\omega,\chi^2(-T+x\cdot\omega) \big)\,\d s 
\right) +O(h^2). 
\end{split}
\ee

(b) For every fixed $\omega\in S^{n-1}$, the second term in the asymptotic expansion of $\Lambda(u_\text{\rm in}^\mathbb{C})$ as $h\to0$  recovers the X-ray transform of $x\mapsto \g(x,p)$  at the direction $\omega$, for every $p\in [0,K^2]$. 

(c) $\Lambda(u_\text{\rm in}^\mathbb{C})$, known for all unit $\omega$, all $0<h\ll1$, recovers $\g(x,p)$ uniquely for all $x$ and $p\in [0,K^2]$. 
\end{theorem}

Rescaling $\chi$ allows us to recover $\g$ for all $x$, $p$ if $f$ is of the type \r{1odd} for all $u$.

We turn our attention now to the  recovery of $f$ with a real-valued incident wave
\be{10real}
u_\textrm{in}^\mathbb{R}= \cos\frac{-t+x\cdot\omega}{h} \chi(-t+x\cdot\omega), \quad \omega\in S^{n-1},
\ee
which is just the real part of \r{10a}, modeling problems where $u$ must be real. Then  the solution will stay real. Consider the odd case \r{1odd} first, then  $f(x,u) = \g(x,|u|^2)u=\g(x,u^2)u$. The incident  wave is a linear combination of waves $\chi(\phi) e^{\i k\phi/h} $ with   $k=-1,1$, and plugging this in the non-linearity would create higher order \textit{odd only} harmonics. To compute them, we expand the principal part of the non-linear term $\g\big(x,\cos^2(\phi/h)\chi^2(\phi)\big) \chi(\phi)\cos(\phi/h)$ into Fourier cosine series in the $\phi/h$ variable, see \r{R1} below. This leads naturally to an expansion of $\g(x,M^2q^2)q$ into Chebyshev polynomials over $q\in[-1,1]$
with  Chebyshev coefficients  
\be{R3}
\gamma_m(x,M) = \frac4\pi \int_{0}^1  \left[\g\big(x,M^2 q^2 \big) q\right] T_m(q) \frac{\d q} {\sqrt{1-q^2}}, 
\ee
where $T_m$ are the Chebyshev polynomials of first kind, and $\gamma_m(x,M)=0$ for $m$ even. Here $M$ is a parameter which will eventually be replaced by $\chi(-T+x\cdot\omega)$. 

\begin{theorem}\label{thm_main_R} 
Let $f$ be odd for $|u|\le K$ as in Theorem~\ref{thm_main}. 
Let the real  $u_\text{\rm in}^\mathbb{R}$ be as in \r{10real}. Then, for $h\ll1$,  there is a unique bounded solution to \r{1} with initial condition \r{10real} for $t<-R-\delta$,  defined for $t\in [0,T]$. Moreover, 

(a)  
\be{R8T}
\begin{split}
\Lambda(u_\text{in}^\mathbb{R}) &=  \cos\frac{-T+x\cdot\omega}{h} \chi(-T+x\cdot\omega)\\
&\quad + h\chi(-T+x\cdot\omega)\sum_{k\ge1,\,\textrm{\rm odd} } \frac1{2k} \sin\frac{k( -T+x\cdot\omega)}{h} X\gamma_k 
+O(h^2),
\end{split}
\ee
where 
\be{R8a}
X\gamma_k(x,\omega)  = \int \gamma_k\big(x+\sigma\omega, \chi(-T+x\cdot\omega) \big)\, \d\sigma. 
\ee

(b) The non-linearity $\g(x,p)$ is uniquely determined by  the second term above in the asymptotic expansion of $\Lambda(u_\text{\rm in}^\mathbb{R})$ as $h\to0$,  for all $x$ and $p\in[0,K^2]$. 
\end{theorem}

To prove the theorem, we show that one can recover $X\gamma_k$ first. Those functions look like the Fourier sine coefficients of the second term in the $x_\parallel :=x\cdot\omega$ variable but they depend on $x_\parallel$ as well, through $X\gamma$. That additional dependence is ``slow'', which allows for the separation. 
 Then we can invert the X-ray transform to get $\gamma_k$, and therefore $\g$. Note that $x_\parallel$ can be fixed here.  We also show that we can recover $\g$ from $\gamma_1$ alone by solving an Abel equation; then we need all $x_\parallel$, see \r{R10}. Details are given in section~\ref{sec_real}. In particular, it is shown there that all steps in the recovery are explicit. 

Finally, when $f$ is not necessarily even in $u$, existence of solution, and the geometric optics construction still works, see section~\ref{sec_G_non-odd} but we get zeroth harmonics. About the inverse problem, we prove the following. 

\begin{theorem}\label{thm_even} 
Let $f= \alpha_1(x)u+ \dots+\alpha_m(x) u^m$ for $|u|\le K$, $m\ge2$. Then, for $h\ll1$, there is unique bounded solution of \r{1} with a real initial condition \r{10real} for $t<-R-\delta$,  defined for $t\in [0,T]$. Moreover,  $\Lambda(u_\textrm{\rm in}^\mathbb{R})$, as $h\to0$, recovers $\alpha_m$ uniquely. In particular, $f(x,u) = \alpha(x) u^m$ is recovered uniquely. 
\end{theorem}
The reconstruction is explicit again and it uses the highest harmonic in the subprincipal term which is not affected by the zeroth one. We make it even more explicit in the special case of a quadratic nonlinearity $f(x,u)=\alpha(x)u^2$. 

Uniqueness in the case $f(x,u) = \alpha(x) u^m$, and stability (for  specific small incident waves) was proven in \cite{lassas2020uniqueness} using the higher order linearization method. The analysis there, as explained earlier, is based on the asymptotic behavior of the non-linearity as $u\to0$. 

The results here can be extended in several directions; we chose not to do so in order to keep the exposition more transparent.  One can involve a Riemannian metric, the non-linearity $f$ can depend on $t$ as well, and one can localize the probing waves on the plane $x\perp\omega$. We refer to section~\ref{sec_remark} for a discussion.

The structure of the paper is as follows. The odd case with a complex incident wave is studied in section~\ref{sec_complex}, where we prove Theorem~\ref{thm_main}. The odd case with a real incident wave and the related geometric optics construction are considered in section~\ref{sec_real}, where we prove Theorem~\ref{thm_main_R}. Section~\ref{sec_even}is devoted to the case of general $f$ and to the proof of Theorem~\ref{thm_even} in particular. Numerical examples are presented in section~\ref{sec-num}, some further remarks can be found in section~\ref{sec_remark}, and a discussion --- in section~\ref{sec_discussion}. In the appendix, we review some known and prove some new results about the solvability of \r{1}.

\section{Odd nonlinearities. A complex monochromatic incident wave} \label{sec_complex}
We assume first $f$ odd as in \r{1odd} and the incident wave complex as in \r{10a}. 
\subsection{Geometric optics} 
Let $u$ be as in Theorem~\ref{thm_main}. We prove first the following. 
\begin{proposition}\label{pr_c} 
Under the assumptions of Theorem~\ref{thm_main}, there is unique bounded solution for $t\in[0,T]$ when $h\ll1$. Moreover, 
\be{main_eq0}
\begin{split}
u(t,x) &= e^{\i (-t+x\cdot\omega)/h} \chi(-t+x\cdot\omega) \\
&\quad \times \left(  
1 - \i \frac{h}2\int_{-\infty}^t  \g\big(x+s\omega,\chi^2(-t+x\cdot\omega) \big)\,\d s 
\right) +O(h^2)
\end{split}
\ee
uniformly in $t$ and $x$.
\end{proposition}

\begin{proof}
As before, we are looking for an asymptotic solution of \r{1} of the form
\be{2}
u=e^{\i\phi(t,x,\omega)/h} a(t,x,\omega,h)  
\ee
with $0\le h\ll1$, $|\omega|=1$, and $a$ having some asymptotic expansion in powers of $h$, not necessarily integer ones a priori: $a=a_0(t,x)+ha_1(t,x)+ \dots$ with $a_j$ independent of $h$.  
Plug \r{2} into \r{1} to get
\be{3}
\begin{split}
e^{-\i\phi /h}P(e^{\i\phi /h}  a) &=  e^{-\i\phi /h}\Box (e^{\i\phi /h} a) +   \g(x, |a|^2)a  \\
&= h^{-2}(-\phi_t^2+|\p_x\phi|^2)a+ 2\i h^{-1}(\phi_t, -\p_x\phi)\cdot(\p_t, \p_x)a\\ &\quad +\i h^{-1} (\Box\phi )a + \Box a+   \g(x, |a|^2)a\\&=0,
\end{split}
\ee
where $P(u)$ is the non-linear operator on the l.h.s.\ of \r{1}.  

We are in the linear regime now. The highest order  $h^{-2}$ term gives us the eikonal equation
\be{5}
\phi_t^2-|\p_x\phi|^2=0.
\ee
The $h^{-1}$ order term gives us the well-known first transport equation, the same as in the linear case:
\be{6}
Ta_0=0, \quad T:= 2 (\phi_t, -\p_x\phi)\cdot (\p_t, \p_x) + \Box\phi. 
\ee
The second transport equation is affected by the non-linearity:
\be{T2}
\i T a_1 + \Box a_0 + \g(x,|a_0|^2)a_0=0.
\ee 
Note that here we used the expansion $\g(x,p+h) =\g(x,p)+O(h)$.  

Now, with the linear phase $\phi = -t+x\cdot\omega$, the first transport equation takes the form
\be{8}
-2 (1, \omega)\cdot (\p_t, \p_x)a_0 = 0.
\ee

Recall that we have the incoming wave \r{10a}. In time-space, for a fixed $\omega$, introduce the variables 
\be{sy}
(s,y)=(t,x-t\omega); \quad \text{then $(t,x)=(s,y+s\omega)$}. 
\ee
Then $\p_s= \p_t+\omega\cdot\p_x$, and $T=-2\p_s$. The first transport equation with its  initial condition takes the form
\be{IC0}
-2\frac{\d}{\d s}a_0 =0, \quad a_0=\chi(y\cdot\omega)\quad \text{for $s\ll0$}. 
\ee
Therefore, 
\be{aoc}
a_0=\chi(y\cdot\omega) = \chi(-t+ x\cdot\omega).
\ee

The second transport equation \r{T2} takes the form
\be{T2a}
-2\i \frac{\d}{\d s} a_1  + \chi(y\cdot\omega)\, \g\big(y+s\omega, 
\chi^2(y\cdot\omega)\big)=0, \quad a_1|_{s\ll0}=0.
\ee 
Note that this would be the transport equation if we had a time dependent potential (depending on the incident direction $\omega$ as well)
\[
V(t,x) = \g(x,  \chi^2(-t+x\cdot\omega)).
\]
The solution of \r{T2a} is
\be{a1}
a_1(s) = -\frac{\i}2 \chi(y\cdot\omega)\int_{-\infty}^s \g\big(y+\sigma \omega, \chi^2(y\cdot\omega)\big)\,\d \sigma.
\ee
If we restrict this to $y\cdot\omega=\tau$, corresponding to $t-x\cdot\omega= \tau$, we get the effective potential $V(x) = \g(x,\chi^2(\tau))$, as explained in the introduction.

Let $u_1=e^{\i\phi/h}(a_0+h a_1)$. Then $u_1$ solves $e^{-\i\phi/h}P(u_1)  =-h\Box a_1 +f(x, e^{-\i\phi/h}(a_0+h a_1))-f(x, e^{-\i\phi/h}a_0) $, see \r{3}. By \r{1odd}, we can write this as $ P(u_1)  =-h e^{\i\phi/h}\Box a_1 +f(x,  (a_0+h a_1))-f(x,  a_0) $, so we get $P(u_1)=O(h)$ in $L^\infty$ but each derivative of the r.h.s.\ multiplies this by $h^{-1}$. By the estimates in Theorem \ref{cut-off}, we get $u_1=O(h)$ in the energy norm, which is not sufficient. 
Following  \cite{S-Antonio-nonlinear}, we continue the construction to a higher order.  Expanding the nonlinearity in Taylor series in $h$, and at each step, we solve linear ODEs. Thus for every $N>1$, we have $u_N$ having an expansion similar to \r{main_eq0} with a remainder $R_N$ satisfying 
\be{R01}
\Box u_N + f(x,u_N) = R_N = O(h^{N-s})\quad \text{in $H^s$}.
\ee
Moreover, $\|u_N\|_{L^\infty}\le C$ for $h<h_0$ for every fixed $h_0>0$ and $C$ depending on $f$ and $\chi$.   Given $\eps>0$, we can choose $h_0$ so that $C=K+\eps$, where $K=\max\chi$. Take $\kappa$ in \r{Weq2} so that $\kappa(q^2)=1$ in a neighborhood of $q^2\le K$ and $\kappa=0$ in a slightly larger one. Then $u_N$ solves \r{R01} with $f$ replaced by its cut-off version $\kappa(u^2)f(x,u)$ as in \r{Weq2}. By Theorem~\ref{cut-off}, there exists a unique $u$ solving \r{1} with $f$ replaced by its cutoff version.  
Choose $s=1$ in \r{R01} to get $u-u_N=O(h^{N-1})$ in $H^2$, uniformly in $t$, by Theorem~\ref{stability}. Then by the trace theorem, this is also true in the uniform norm since $n=2,3$. In particular, this shows that $|u|$ has the same upper bound up to $O(h)$; therefore, we can remove the cut-off to get that $u$ actually solves \r{1}. By 
Theorem~\ref{unique-1}, this  solution $u$ is unique among all bounded solutions for $t\in [0,T]$.  
\end{proof}

\begin{remark}
The statement of the proposition holds for $0<h\le h_0$ with $h_0$ depending on $\chi$ and on $f$. For the purpose of the unique recovery statements in the main theorems, one can take $f_1$ and $f_2$; apply the expansion with the smallest of the two $h_0$'s, and then prove $f_1=f_2$. 
\end{remark}

We should also remark here that  \cite[Theorem 8.3.1]{Rauch-geometric-optics} shows that, for semilinear systems of order one, at least, once one finds an approximate oscillatory solution $u_h$ as above, up to order $h^\infty,$ then there exists an exact solution $u$ which is $C^\infty$  and  such that $u-u_h=O(h^\infty)$.  


\begin{proof}[Proof of Theorem~\ref{thm_main}] 
The existence and the uniqueness statement and part (a) follow from Proposition~\ref{pr_c}. 

Fixing $-T+x\cdot\omega= :\tau\in \supp\chi$, we recover the X-ray transform of $\g(\,\cdot\,,K^2\chi^2(\tau))$ at the direction $\omega$. Indeed, writing $x=z+\sigma\omega$, $z\perp\omega$, we get
\[
\begin{split}
\int  \g\big(x+s\omega,\chi^2(-T+x\cdot\omega)\big)\,\d s &= \int  \g\big(z+(s+\sigma) \omega,\chi^2(-T+\sigma)\big)\,\d s\\
&= \int  \g\big(z+s \omega,\chi^2(-T+\sigma)\big)\,\d s. 
\end{split}
\]
When $|T-x\cdot\omega|\le\delta$, we have $|T-\sigma|\le\delta$, so when we fix $\tau = -T+x\cdot\omega= -T+\sigma$, we have $|\sigma|\le\delta$ and the formula above is the X-ray transform of $\g(\cdot, \chi^2(\tau))$ for such $\tau$. 
 Varying $\tau$, $\chi^2(\tau)$ runs over the range of $\chi^2$ which is $[0,K^2]$. This proves (b). Part (c) follows immediately. 
\end{proof}

\section{Odd nonlinearities. A real incident monochromatic wave} \label{sec_real}
We assume $f$ odd as in \r{1odd} and the incident wave real as in \r{10real}. 
\subsection{Heuristic arguments} 
Now, we would expect $|u|^2 \sim \cos^2(\phi/h)\chi^2(\phi)$, $\phi=-t+x\cdot\omega$, i.e., the principal term in the expansion of $u$ to be unaffected by the non-linearity but that would change in next section. Plugging this into $\g(x,|u|^2)$ yields the formal potential
\be{R_NL}
V:=\g\big(x,\cos^2(\phi/h)\chi^2(\phi)\big) 
\ee
 (up to an $O(h)$ remainder), which oscillates highly, therefore the method we used before needs modifications. The nonlinearity  $\textrm{NL}= \g(x,|u|^2)u$ then takes the form
\be{R_NL1}
\textrm{NL}=\g\big(x,\cos^2(\phi/h)\chi^2(\phi)\big) \chi(\phi)\cos(\phi/h)
\ee
modulo $O(h)$.  We think of $\phi$ in $\chi(\phi)$ as a second copy of $\phi$, independent of the one in the $\cos(\phi/h)$ term. Set $\chi(\phi)=:M$.  
We may want to use the formula $2\cos^2\theta=1+\cos(2\theta)$. 
Then  $\text{NL} =g(x, M^2\cos^2(\phi/h)) M\cos(\phi/h)$. Let us expand the $2\pi$ periodic, even function $\theta\mapsto \g\big(x,M^2 \cos^2\theta \big)\cos\theta$ (note that we removed the factor $M$) into a Fourier cosine series
\be{R1}
\g\big(x,M^2\cos^2 \theta \big)\cos\theta  =\frac{\gamma_0(x,M)}2+ \sum_{m\ge1} \gamma_m(x,M)\cos (m\theta),
\ee
(we will see that $\gamma_0=0$ in a moment and that $m$ must be odd), where
\be{R5}
\gamma_m(x,M) = \frac2\pi \int_0^\pi   \g\big(x,M^2\cos^2(\theta) \big) \cos\theta  \cos(m\theta) \,\d\theta.
\ee
Perform the change of variables $ q= \cos\theta$ to get formula \r{R3} stated in the Introduction. Knowing $\gamma_m$, we can recover $\g(x,p)$ for $0\le p\le M^2$ by
\be{R3a}
\g(x,M^2q^2) q = \sum \gamma_m(x,M)T_m(q).
\ee
The non-linearity takes the form
\be{R4}
\textrm{NL}= \chi(\phi) \sum_{m\ge0, \, \text{\rm odd}} \gamma_m(x, \chi(\phi) )\cos(m\phi/h)
\ee
modulo $O(h)$.

\subsection{Geometric optics in the principally linear regime} 
 Assume 
\be{R2}
u\sim  \sum_{k\in \mathbf{Z}\setminus 0}  e^{\i k \phi/h}  a^{(k)}(t,x,\omega,h),
\ee
compare with \r{2}, where each $a^{(k)}$ has an expansion $a^{(k)} = a^{(k)}_0+h a^{(k)}_1+\dots$. 
Since we want $u$ to stay real, we require $\bar a^{(k)}= a^{(-k)}$. Using \r{3}, we see that we can take the phase $\phi$ satisfying the same eikonal equation \r{5}; and we take the phase $\phi=-t+x\cdot\omega$ because this is what is dictated by the initial condition \r{10real}. Plug \r{R1} into the equivalent of \r{3} (no absolute value in $|u|^2$ there) to get 
\be{exp}
\begin{split}
-h^{-1}\sum_k & e^{\i k\phi/h}2\i k\frac{\d}{\d s} \left(  a^{(k)}_0+h a^{(k)}_1+\dots\right) \\
& +\g \Big(x, \Big( \sum_m e^{\i m\phi/h} \Big( a^{(m)}_0+h a^{(m)}_1+\dots\Big)\Big)^2 \Big)\sum_k e^{\i k\phi/h}    \Big( a^{(k)}_0+h a^{(k)}_1+\dots\Big) \\
&+ \sum_k e^{\i k\phi/h}  \Box \left(  a^{(k)}_0+h a^{(k)}_1+\dots\right)=0.
\end{split}
\ee
Expand $\g$ into a Taylor series in its second variable to get that the second term above multiplied by $h$ equals
\be{R6}
h\g \Big(x, \Big( \sum_m  e^{\i m\phi/h}    a^{(m)}_0 \Big)^2 \Big) \sum_k e^{\i k\phi/h} a^{(k)}_0   + O(h^2).
\ee
The first transport equations, see also \r{IC0}, says that $a_0^{(k)}$ stay constant along the rays, therefore, by \r{10real},
\be{R6a}
a^{(1)}_0=a^{(-1)}_{0} = \frac12 \chi(y\cdot\omega),
\ee
in the variables \r{sy}, and all other $a^{(k)}_0$ coefficients vanish. In particular, this means that in \r{R6}, $m=-1,1$ only, same for $k$ there. Then  \r{R6} takes the form
\[
h \g(x, M^2\cos^2(\phi/h)) M\cos(\phi/h) +O(h^2), 
\]
which is just $hM \textrm{NL}+O(h^2))$, see \r{R_NL1}, as expected. 
We expanded this into Fourier cosine series in \r{R4}. This, together with \r{exp}, shows that the second transport equations take the form
\be{R6b}
-2\i k \frac{\d}{\d s}a_1^{(k)}  + \frac12 \gamma_k(y+s\omega, \chi( y\cdot\omega) ) \chi( y\cdot\omega)   =0, \quad a_1^{(k)} |_{s\ll0}=0
\ee
with the Chebyshev coefficients $\gamma_k$ given by \r{R3}. 
Therefore,
\be{R7}
a_1^{(k)}(s) = -\frac{\i}{4k} \chi( y\cdot\omega) \int_{-\infty}^s \gamma_k\big(y+\sigma\omega, \chi(y\cdot\omega)\big)\, \d\sigma,
\ee
compare with \r{a1}. The coefficients $\gamma_m$ are rapidly converging, locally uniformly with respect of its variables, which makes it easy to prove convergence in \r{R8}. The construction can be continued up to any finite order with a justification of the expansion as in the previous section. 

This proves the following.

\begin{proposition}\label{pr_R} 
Under the assumptions of Theorem~\ref{thm_main_R}, there is unique bounded solution $u$ defined for $t\in [0,T]$. Moreover, 
\be{R8}
\begin{split}
u&=  \cos\frac{-t+x\cdot\omega}{h} \chi(-t+x\cdot\omega)\\
&\quad + h\chi(-t+x\cdot\omega)\sum_{k\ge1,\,\textrm{\rm odd} }\frac1{2k}  \sin\frac{k( -t+x\cdot\omega)}{h} \int_{-\infty}^t \gamma_k\big(x+\sigma\omega, \chi(-t+x\cdot\omega) \big)\, \d\sigma\\&\quad +O(h^2),
\end{split}
\ee
 uniformly, where the Chebyshev coefficients of the non-linearity are given by \r{R3}.
\end{proposition}

\subsection{The inverse problem. Proof of Theorem~\ref{thm_main_R}} 
By Proposition~\ref{pr_R}, 
\[
 \begin{split}
\Lambda (u_\text{\rm in}^\mathbb{R}) &=   \chi(-T+x\cdot\omega) \bigg( \cos\frac{-T+x\cdot\omega}{h}
 + h \sum_{k\ge1,\,\textrm{\rm odd} } \sin\frac{k( -T+x\cdot\omega)}{h} X\gamma_k\bigg) +O(h^2),
\end{split}
\]
with $X\gamma_k$, given by \r{R8a},  
is the X-ray transform of $\gamma_k$ in the $x$ variable. In particular, this proves part (a) of Theorem~\ref{thm_main_R}. 
 It is convenient to set 
\[
x=x_\perp+x_\parallel\omega,
\]
where $x_\perp=x-(x\cdot\omega )\omega \perp \omega$, $x_\parallel=x\cdot\omega$. Then 
\[
 \begin{split}
\Lambda (u_\text{\rm in}^\mathbb{R}) &=   \chi(-T+x_\parallel) \bigg( \cos\frac{-T+x_\parallel}{h}
 + h \sum_{k\ge1,\,\textrm{\rm odd} } \sin\frac{k( -T+x_\parallel)}{h} X\gamma_k\bigg) +O(h^2),
\end{split}
\]
with
\be{R11b}
X\gamma_k(x,\omega) = \int \gamma_k\big(x_\perp+\sigma\omega, \chi(-T+x_\parallel) \big)\, \d \sigma. 
\ee
Given $\Lambda(u_\text{in}^\mathbb{R})$, we know the subprincipal term  above. That term looks like a Fourier sine series in the $x_\parallel$ variable with Fourier coefficients $X\gamma_k$ but $X\gamma_k$ depends on $x_\parallel$ as well. 
On the other hand, it does not oscillate fast with $h$, so in any interval of length $\sim h$ is constant up to $O(h)$. Therefore, we can  compute the Fourier coefficients for $x_\parallel$   in any fixed interval with length the period $2\pi h$ as if $X\gamma_k$ did not depend on $x_\parallel$, up to $O(h)$. This shows that we can recover $X\gamma_k$. 

We have recovered the X-ray transform of $  \gamma_k\big(\,\cdot\, , \chi(-T+x_\parallel) \big)$ in the direction $\omega$. Now we vary $\omega$ and take measurements at $x$ with $x_\parallel= x\cdot\omega$ unchanged (we can think of it as rotating the setup). Then we recover the Chebyshev coefficients $\gamma_k(x,M)$ for every $M=\chi(-T+x\cdot\omega)$. Varying $x$ in the strip $|x\cdot\omega-T|\le\delta$ allows is to recover $\gamma_k(x,M)$ for every $M\in [0,K]$ with $K=\max\chi$ as before. Then we can recover $\g(x,p^2)p$ for $|p|\le K$, see \r{R1}. 

One downside of this is that to recover all coefficients, we need measurements at increasing frequencies $k/h$, $k=1,3,\dots$. If we are limited in that, we can recover some finite Fourier expansion only. When we do numerical simulations, higher $k$ means even smaller step sizes if we use finite differences.

There is an alternative way, however. We can recover the first Chebyshev coefficient $\gamma_1(x,M)$ only. By \r{R3}, we know
\be{R9}
\begin{split}
 \gamma_1(x,M) &= \frac4\pi  \int_{0}^1   \frac{\g(x,M^2q^2) q^2  } {\sqrt{1-q^2}}\, \d q\\
&= \frac4{M^3 \pi} \int_{0}^M  \frac{\g(x,q^2)q^2  } {\sqrt{1-q^2/M^2}}\, \d q\\
&=  \frac4{M^2\pi} \int_{0}^M  \frac{\g(x,q^2) q^2 } {\sqrt{M^2-q^2}} \, \d q. 
\end{split}
\ee
This is an Abel equation with an explicit unique solution \cite[p.~24]{Gorenflo_Abel}
\be{R10}
\g(x,q^2) = \frac1{2q^2} \frac{\d}{\d q} \int_{0}^q  \frac{M^3 \gamma_1 (x,M) } {\sqrt{q^2-M^2}}\,  \d M.
\ee

To compensate for noise, in order to extract the Fourier coefficients in \r{R13}, 
we may want to integrate over a larger interval of ``small'' length but independent of $h$. Next, division by $\chi$ is not really needed before integration. Based on that, we propose the following scheme. Denote by $u_L(T,x,\omega)$ the principal term above. Note that this is exactly the (linear) solution if there were no nonlinearity, i.e., when $\g=0$. 
Compute (denoting $x_\parallel=s$)
\be{R11x}
A_k := h^{-1}\int \sin\frac{k(-T+s)}{h} \big(\Lambda (u_\text{\rm in}^\mathbb{R})(x_\perp+s\omega,\omega)-u_L(T,x_\perp+s\omega,\omega) \big)\psi(\sigma-s)\d s
\ee
where $\psi\in C_0^\infty(\R)$. 
In other words, we multiply the difference $ \Lambda (u_\text{\rm in}^\mathbb{R})-u_L$ by $\psi(\sigma-s)$, the variable $\sigma$ would control the shift; and then, and then we project on a sine Fourier mode.  
 
By \r{R11x},
\be{R13}
\begin{split}
A_k &= \sum_{k'\ge1,\,\textrm{\rm odd} } \int  \sin\frac{k(-T+s)}{h} \sin\frac{k' ( -T+s)}{h}\chi(-T+s) X\gamma_{k'} ( x_\perp+s\omega,\omega)\psi(\sigma-s)\d s + O(h)\\
&=  \sum_{k'\ge1,\,\textrm{\rm odd} } \int  \sin\frac{ks}{h} \sin\frac{k' s}{h}\chi(s) X\gamma_{k'} ( x_\perp +(s+T)\omega,\omega)\psi(\sigma-s-T)\, \d s + O(h).
\end{split}
\ee
Set $f_k(s) = \chi(s)X\gamma_k ( x_\perp+(s+T)\omega,\omega)\psi(\sigma-s-T)$ for a moment, suppressing the other variables. Use the formula $2\sin(ks/h)\sin(k' s/h)= \cos((k'-k)s/h)) -\cos((k'+k)s/h)$ to get 
\[
\begin{split}
A_k =\frac12 \sum_{k'\ge1,\,\textrm{\rm odd} }  \Re \left(\hat f_k\Big( \frac{k'-k}{h}\Big) -   \hat f_k\Big( \frac{k'+k}{h}\Big)\right) =  \frac12  \int   f_k(s) \,\d s + O(h)
\end{split}
\]
since the only non $O(h^\infty)$ contribution comes from $\hat f_k(0)$. 
Indeed, by \r{R5}, $|\partial_x^\alpha \gamma_k|\le C_\alpha$ independently of $k$. Then $|\hat f_k(\xi)|\le C_N\langle\xi \rangle^{-N}$, $\forall N$, independently of $k$.  Then
\[
\begin{split}
\Big|A_k-\frac12 \hat f_k(0)\Big|&\le C_N \sum_{k'\ge1,\,\textrm{\rm odd}, \, k'\not=k}  \left( \Big\langle \frac{k'-k}{h}\Big\rangle^{-N} +  \Big\langle \frac{k'+k}{h}\Big\rangle^{-N} \right)+ O(h)\\  
&\le C_N'\sum_{m\ge1}\langle m/h\rangle^{-N} + O(h)
\end{split}
\]
We have $\langle m/h\rangle^{-1} =(1+(m/h)^2)^{-1}= h(h^2+k^2)^{-1/2}< h/m$ for $m\ge1$. Therefore, the sum above can be estimated by $C_N'h^N\sum_{k\ge1}m^{-N}\le C_N'' h^N$ for $N\ge2$. 

We proved the following.

\begin{proposition} \label{pr_Xg}
We have 
\be{R14}
A_k = \int  \chi(s)X\gamma_k ( x_\perp+(s+T)\omega,\omega)\psi(\sigma-s-T)\,\d s +O(h)
\ee
In particular, choosing a sequence of $\psi$ converging to the Dirac $\delta$, we can recover 
\be{R14a}
\chi(\sigma-T)X\gamma_k ( x_\perp+\sigma\omega,\omega)
\ee
for every $\sigma$, $\omega$ and $x_\perp$. 
\end{proposition}

Therefore, we can recover the convolution of $s\mapsto \chi(s)X\gamma_k ( x_\perp+(s+T)\omega,\omega)$ with arbitrary test functions. In particular, we can recover $X\gamma_k ( x_\perp+\sigma\omega,\omega)$ (think of $\sigma$ as $x_\parallel$) as long as $\sigma-T\in\supp\chi$ but the latter is guaranteed when $X\gamma_k ( x_\perp+\sigma\omega,\omega)$ does not vanish, by \r{R8a}. The convenience of \r{R14} is that it allows us to recover a convolved (a regularized) version of the latter directly from possibly noisy data.

\begin{remark}
To recover $X\gamma_k$ we had to take two limits: first $h\to0$, and next, $\psi\to\delta$ in $\mathcal{D}'$. In applications, we would like take $h\ll1$ and get $X\gamma_k$ up to a small error. One can set $\psi_h(s) =h^{-\mu} \psi(s/h^\mu)$ with $0<\mu<1$ in \r{R11x}. Then  $A_k$ would be equal to  \r{R14a} directly  plus $o(1)$ as $h\to0$, and one can be more specific about the error. 
\end{remark}

\subsection{Recovery algorithms}
We deduced two recovery algorithms. 

\subsubsection{Recovery using all harmonics} 

\begin{itemize}
\item[(i)] Fix $\chi$ and recover $X\gamma_k(x,\omega)$ for $k=1,3,\dots$, by Proposition~\ref{pr_Xg}, and all $x$, $\omega$. 
\item[(ii)] Invert the X-ray transform for every $x_\parallel$ fixed, see \r{R11b}, to recover $\gamma_k(x,M)$ for every $x$ and for every $M\in [0,K]$, $K:=\max\chi$.
\item[(iii)] Recover $\g(x,p)$ by \r{R3a} for every $x$ and every $p\in [0,K^2]$. 
\item[(iv)] To recover $\g(x,p)$ for $p$ on a larger interval, we choose another $\chi$, for example we replace the original one by a scaled version $C\chi$. 
\end{itemize}

This requires measurements or numerical experiments for increasing frequencies $k/h$, $k=3,5,\dots$. 
Note that we actually need one value of $M$ to recover $\g$ for $p$ in a fixed interval; in fact choosing $M=K$ is enough. This corresponds to making measurements at a fixed $x_\parallel$ which maximizes $\chi(x_\parallel-T)$.

\subsubsection{Recovery through the first harmonic} 
\begin{itemize}
\item[(i)] Fix $\chi$ and recover $X\gamma_1(x,\omega)$ by Proposition~\ref{pr_Xg}
\item[(ii)] Invert the X-ray transform for every $x_\parallel$ fixed, see \r{R11b}, to recover $\gamma_1(x,M)$ for every $x$ and for every $M\in [0,K]$, $K:=\max\chi$.
\item[(iii)] Solve the Abel equation \r{R9} to recover $\g(x,p)$ by \r{R10} for every $x$ and every $p\in [0,K^2]$. 
\item[(iv)] To recover $\g(x,p)$ for $p$ on a larger interval, we choose another $\chi$, for example we replace the original one by a scaled version $C\chi$. 
\end{itemize}

Now, we use  all $x_\parallel$ in the data but the first harmonic is enough.

\subsection{The polynomial (cubic) case} \label{sec_cubic}
As a special case, we show how our approach works when $\g(x,|u|^2)u=\alpha(x) u^3$ (and $u$ is real-valued), i.e., when $\g(x,p)=\alpha(x)p$. The construction works in the same way for $\g(|u|^2)u=\alpha(x) u^{2j+1}$, $j\ge1$ integer. 
This is the case considered in \cite{S-Antonio-nonlinear} but there we used  weakly non-linear solutions with amplitudes $\sim h^{-1/2}$. The non-linearity there affects the principal term and creates phase shifts. 

 Then, see \r{R1}, 
\[
\g(x,M^2\cos^2(\theta)) =M^2 \alpha(x) \left(\frac12  + \frac12\cos(2\theta)\right),
\]
which leads to the effective potential, compare to \r{R_NL}, 
\[
V=\alpha(x) \chi^2(\phi) \left(\frac12  + \frac12\cos(2\theta)\right).
\]
Therefore, after multiplying by $u$, we get first and third harmonics only in this case. Indeed, 
\be{Rp0}
\g(x, M^2\cos^2\theta)\cos\theta = M^2\alpha \Big( \frac34\cos \theta+\frac14\cos(3\theta)\Big),
\ee
see also \r{R1}, therefore, $\gamma_1 =3M^2\alpha/4 $, $\gamma_3= M^2\alpha/4 $, all other vanish.

 This leads to an ansatz of the type \r{R2} with $a_1^{(k)}$ having non-zero entries for $k\in \{-3,1,1,3\}$ only:
\be{Rp1}
\begin{split}
u &\sim a_0^{(-1)}e^{-\i\phi/h} + a_0^{(1)}e^{\i\phi/h} \\
 & \quad + h\Big(  a_1^{(-3)}e^{-\i\phi/h} + a_1^{(-1)}e^{\i\phi/h} + a_1^{(1)}e^{\i\phi/h} + a_1^{(3)}e^{3\i\phi/h}          \Big) + O(h^2)\\
 &= 2\Re\Big( a_0^{(1)}e^{\i\phi/h} + h\Big( a_1^{(1)}e^{\i\phi/h} + a_0^{(3)}e^{3\i\phi/h}          \Big)\Big) + O(h^2).
\end{split}
\ee
The $\sim h^2$ term would have harmonics $k\in \{-9,-7,\dots,7,9\}$, etc. As before, the principal terms is unaffected by the non-linearity, so we still have \r{R6a}, therefore the principal part is $u_0:=\chi(y\cdot\omega) \cos(\phi/h)$ in the coordinates \r{sy}. Then, as in \r{Rp0},
\[
\alpha u^3 = \alpha\Big( \frac34\cos(\phi/h)+\frac14\cos(3\phi/h)\Big)\chi^3(y\cdot\omega) ,
\]
which has Fourier coefficients  $\{1/8,3/8, 3/8, 1/8\}$ all multiplied by $M^3\alpha$ with $M$ as above. They can be written as $3M^3 \alpha/(8 |k|)$.

By \r{exp}, the second transport equations are
\[
-2\i k \frac{\d}{\d s} a_1^{(k)} + \frac3{8|k|} M^3\alpha(y+s\omega)  =0, \quad k=-3,-1,1,3. 
\]
The zero initial conditions imply
\be{Rp3}
a_1^{(k)} =- \frac{3\i} {16 k^2} M^3\sign(k) \int   \alpha(y+s\omega)\,\d s.
\ee
We compare this with \r{R7}. We have 
\[
\gamma_1(y+\sigma\omega,\chi(y\cdot\omega)) = \frac34  \chi^2(y\cdot\omega)\alpha(y+\sigma\omega), \quad \gamma_3(y+\sigma\omega,\chi(y\cdot\omega)) = \frac14  \chi^2(y\cdot\omega)\alpha(y+\sigma\omega).
\]
Multiply this by ${-\i /(4k)}$ 
and integrate in $\sigma$ to get 
\[
a_1^{(1)}(s) = -\frac{3\i}{16}  \chi^3(y\cdot\omega)  \int   \alpha(y+s\omega)\,\d s, \quad a_1^{(3)}(s)=\frac19 a_1^{(1)}(s);
\]
also, $a_1^{(-1)}= -a_1^{(1)}= \bar a_1^{(1)}$, $a_1^{(-3)}= -a_1^{(3)}= \bar a_1^{(3)}$, as expected. This confirms \r{Rp3}. 

The subprincipal term in the expansion \r{R8T} of $u$ therefore is
\be{Rp3a}
\frac{h}{8}\chi^3(\phi) X\alpha \bigg(\sin(\phi/h)+\frac19 \sin(3\phi/h)   \bigg ).
\ee
The Abel equation \r{R9} involving $\gamma_1$ takes the form
\[
\gamma_1 = \frac4{\pi} \int_0^M\frac{\alpha(x)M^2q^4}{\sqrt{1-q^2}}\,\d q , 
\]
and a direct computation yields $\gamma_1=3M^2\alpha/4$, as we found out earlier. The recovery formula \r{R10} takes the form
\be{Rp4}
\alpha q^2 = \frac1{2q^2} \frac{\d}{\d q} \int_{0}^q  \frac{M^3 (3M^2\alpha/4 ) } {\sqrt{q^2-M^2}}\,  \d M.
\ee
Set $M=q\cos\theta$ in the integral to  transform the right-hand side to
\[
\frac{3\alpha} {8q^2} \frac{\d}{\d q} q^5\int_{0}^{\pi/2}   {\cos^5\theta }  \,  \d \theta = \frac{3\alpha} {8q^2} \frac{\d}{\d q} q^5 .\frac{8}{15} = \alpha q^2,
\]
which confirms \r{Rp4}. 

\section{Not necessarily odd non-linearities} \label{sec_even}

Assume now that the non-linearity is $f(x,u)$, not necessarily odd in $u$. 

Some heuristic arguments can convince us that we cannot expect the zeroth harmonic to affect the subprincipal (and lower) terms only. If we follow \r{exp}, we would see that when $k=0$, we are missing the first transport equation for the zeroth harmonic leading term $a_0^{(0)} $ because that equation would be multiplied by $-2\i k$. Similarly, in \r{R6b}, the derivative cancels when $k=0$, etc. For this reason, we will seek a zeroth harmonic contribution to the \textit{principal} term.

\subsection{A quadratic non-linearity}\label{sec_q2}
We start with the example $f=\alpha(x)|u|^2$ which is interesting on its own. 
Note that global solvability is not guaranteed and probably not even true. On the other hand, solutions of the type we need do exist, as it follows form our analysis. 
Assume the real incident wave \r{10real} as before. 
 We will look for a solution of the form
\be{E01}
\begin{split}
u&= \frac{\chi(\phi)} 2\Big(e^{-\i\phi/h}+ e^{\i\phi/h}\Big)+ u_0^{(0)}\\
&\quad + h\Big(e^{-2\i\phi/h} a_1^{(-2)}+ e^{-\i\phi/h} a_1^{(-1)}+  u_1^{(0)} +e^{\i\phi/h} a_1^{(1)}+ e^{2\i\phi/h} a_1^{(2)}\Big)+ O(h^2),
\end{split}
\ee
where the zeroth harmonic is $u^{(0)}= u^{(0)}_0+h u^{(0)}_1+\dots$; a more consistent notation would be $a^{(0)}$. 
We presume that the $k=\pm1$ harmonics would be the same as those of $u_\textrm{in}^\mathbb{R}$, and in fact we get that directly from the leading transport equations for $k\not=0$. 
Plugging this into the quadratic non-linearity, we get a principal part
\be{E02}
 \text{NL}_0:=  \frac12  \alpha {\chi^2(\phi)}(1+\cos(2\phi/h)) + \alpha \Big(u_0^{(0)} \Big)^2 + 2\alpha{\chi(\phi)} u_0^{(0)} \cos(\phi/h).
\ee
The zeroth harmonic is
\be{E02a}
 \frac12\alpha {\chi^2(\phi)} +  \alpha \Big(u_0^{(0)} \Big)^2  . 
\ee
Plugging this in \r{1} and isolating the zeroth harmonics in the principal term,  we can expect the principal part of the zeroth harmonic to solve
\be{E03}
\Box u_0^{(0)} + \alpha \Big(u_0^{(0)} \Big)^2  = -\frac12\alpha(x) \chi^2(-t+x\cdot\omega), \quad u_0^{(0)} |_{t\ll0}=0.
\ee
Next we determine $a_1^{(k)}$ for $k=-2,-1,1,2$. The transport equation for $a_1^{(1)}$ is
\be{E03_1}
-2\i \frac{\d}{\d s} a_1^{(1)}    =-
\alpha \chi(\phi) u_0^{(0)} , \quad  a_1^{(1)} |_{t\ll0}=0.
\ee
In order to obtain it, we had to select the $k=1$ harmonic in \r{E02} to put on the r.h.s. Since $u_0^{(0)} $ is already determined, we can integrate it along the characteristic to get $a_1^{(1)}$. Then $a_1^{(-1)}$ is just its conjugate. In particular, the factor $\i=\sqrt{-1}$ shows that the first harmonic in the subprincipal term has a sine term. For $k=2$, we get similarly
\[
-4\i \frac{\d}{\d s} a_1^{(2)}    =-\frac14 \alpha \chi^2(\phi)  , \quad  a_1^{(2)} |_{t\ll0}=0.
\]
The solution is
\[
a_1^{(2)}(s) =- \frac{\i}{16}\chi^2(y\cdot\omega)\int_{-\infty}^s \alpha(y+\sigma\omega)\,\d\sigma,
\]
compare with \r{R7}. In the $(t,x)$ variables, see \r{sy}, we get
\be{E04}
a_1^{(2)}(t,x) =- \frac{\i}{16}\chi^2(-t+x\cdot\omega)\int_{-\infty}^0 \alpha(x+\sigma\omega)\,\d\sigma,
\ee
and in particular, 
\be{E05}
a_1^{(2)}(T,x) =- \frac{\i}{16}\chi^2(-T+x\cdot\omega)X\alpha(x,\omega),
\ee
for $x$ as in \r{2a}. Therefore, this term recovers the X-ray transform of $\alpha$.

To get  the equation for  $u_1^{(0)}$ (the second order term of the zeroth harmonic), we need to compute the non-oscillating terms in the $\sim h$ term of the non-linearity, i.e., of twice the product of the $O(1)$ and the $O(h)$ term in \r{E01}. They can be obtained by combining $k=0$ and $k=0$ in each term; or $k=1$ with $k=-1$, respectively $k=-1$ and $k=1$. 
We get
\be{E05a}
\Box u_1^{(0)} + 2\alpha u_0^{(0)} u_1^{(0)} =- \alpha \chi^2(\phi) a_1^{(1)}, \quad u_1^{(0)}|_{t\ll0}=0.   
\ee
This is a linear wave equation for $u_1^{(0)}$ with known coefficients. 

One can compute a full asymptotic expansion that way. The next order term will have harmonics $\{-4,-3,\dots,3,4\}$, etc. We proved (the justification is as in the odd case) the following.

\begin{proposition}
The unique bounded solution to 
\be{E06}
\Box u + \alpha(x)u^2=0
\ee
with initial condition \r{10real}  satisfies
\be{E07}
\begin{split}
u &= \chi(-t+x\cdot\omega) \cos\frac{-t+x\cdot\omega}h + u_0^{(0)}\\
&\quad  + h u_1^{(0)} + 2\i h  a_1^{(1)} \sin\frac{-t+x\cdot\omega}h + 2\i h  a_1^{(2)} \sin\frac{2(-t+x\cdot\omega)}h + O(h^2).
\end{split}
\ee
\end{proposition}

Consider the inverse problem for \r{E06} now: recovering $\alpha$ from $\Lambda(u_\textrm{in})$. The principal term in \r{E07}, at $t=T$ and $x$  as in \r{2a}, carries information about $\alpha$ in the zeroth harmonic $u_0^{(0)}$.  The latter solves the non-linear wave equation \r{E03} however, so recovery of $\alpha$ from it seems not easier than the original problem. Also, since $u_0^{(0)}$ is smooth, no additional parameter, it would carry no meaningful resolution about $\alpha$ anyway. Next, the term $u_1^{(0)}$ has the same problem; so we look at the two oscillatory terms. The amplitude $2\i a_1^{(1)}$ of the first harmonic solves \r{E03_1} which depends on $u_0^{(0)}$ along the characteristic, which we do not know. The amplitude $2\i a_1^{(2)}$ of the second harmonic however is given by \r{E05}, which recovers $X\alpha$ explicitly. The next question is whether we can recover $ a_1^{(2)}$ from \r{E07}. This can be done as in Proposition~\ref{pr_Xg} since the contribution from the zeroth harmonics to $A_k$ in \r{R11x} would contribute an $O(h^\infty)$ term to its asymptotic. We provide more details below.

In other words, the inverse problem in this case is solvable in an explicit way as well.

\subsection{General non-odd non-linearities} \label{sec_G_non-odd}
When we plug $u=u_0+\chi(\phi) \cos(\phi/h)$ into the non-linearity, we get $f(x,u_0+\chi(\phi) \cos(\phi/h))$. Following the special case above, we expect the following Fourier coefficients to play a role:
\be{Ef}
\mathsf{f}_k(x,M, u)=\frac2{\pi} \int_{0}^\pi f(x,u+M \cos\theta) \cos(k\theta)\, \d\theta,
\ee
where we will eventually set $M=\chi(\phi)$ as done earlier. The zeroth coefficient $\mathsf{f}_0$ would be responsible for the leading term of the zeroth harmonic. 
When $f=\alpha u^2$, for example, we get $\mathsf{f}_0 = \alpha M^2+2\alpha u^2$, see \r{E02a}.  We are looking for $u$ having an asymptotic expansion 
 \be{E11}
\begin{split}
u&= \frac{\chi(\phi)} 2\Big(e^{-\i\phi/h}+ e^{\i\phi/h}\Big)+ u_0^{(0)} +h u_1^{(0)} + h\sum_{k}  e^{\i k\phi/h} a_1^{(k)}  + O(h^2),
\end{split}
\ee
Then  $u_0^{(0)}$ must solves  
\be{E12}
\Box u_0^{(0)} + \frac12 \mathsf{f}_0\big(x,\chi(\phi) , u_0^{(0)}\big)=0, \quad u_0^{(0)}|_{t\ll0}=0.
\ee
In general, $\mathsf{f}_0\not=0$ when $u=0$; and in fact, this is true  if and only if $f_0$ is odd in the $u$ variable for $|u|\le M$. Then, by uniqueness, we would get  $u_0^{(0)}=0$ in \r{E12}, which is what we got in the odd case. 

The equivalent of \r{E02} now is
\[
\text{NL}_0 = \sum \mathsf{f}_k\big(x,\chi(\phi),u_0^{(0)}\big) \cos(k\phi/h).
\]
Plugging \r{E11} into the PDE \r{1}, and arguing as in \r{exp}, we see first that \r{E12} is justified; and 
\be{E13}
2\i k \frac{\d}{\d s} a_1^{(k)} +  \mathsf{f}_k( x,\chi(\phi),u_0^{(0)}  ) =0,\quad k\not=0.
\ee
The equation for the subprincipal term $u_1^{(0)}$ of the zeroth mode looks similar to \r{E05a} with coefficients obtained by averaging the $\sim h$ term of the Taylor expansion of $f(x,\cdot)$ with the dot there replaced by \r{E11}. The construction can be extended to higher order and justified as done earlier. 

\subsection{The inverse problem}  The appearance of the zeroth modes complicates the inverse problem. They depends on the non-linearity in an implicit non-linear way. One special case when we can recover the non-linearity is when it is of the kind $f=\alpha(x)u^{2m}$ with $m\ge1$ integer, inspired by the quadratic case in section~\ref{sec_q2}. 
 
We show first that the coefficients in \r{E11} are recoverable from the data.

\begin{proposition} \label{pr_Xg_even}
For $A_k$ defined in \eqref{R11x}, we have
\be{E14}
A_k = \int    a_1^{(k)} (T , x_\perp+(s+T)\omega,\omega)\psi(\sigma-s-T)\,\d s +O(h), \quad k\ge1.
\ee
In particular, choosing a sequence of $\psi$ converging to the Dirac $\delta$, we can recover 
\[
 a_1^{(k)} (T,  x_\perp+\sigma\omega,\omega)
\]
for every $\sigma$, $\omega$ and $x_\perp$. 
\end{proposition}

\begin{proof}
The proof is as that in Proposition~\ref{pr_Xg}. What is new here is that we have   zeroth modes but they would give us an $O(h^\infty)$ contribution. 
\end{proof}

\begin{proof}[Proof of Theorem~\ref{thm_even}]
In this case, $\mathsf{f}_m = 2^{1-m} M^m \alpha_m$ which can be seen easily by expanding $(u+M\cos\theta)^m$ by the binomial theorem, and then in Fourier series. Indeed, the highest power of $\cos\theta$ would be $M^m \cos^m\theta = M^m2^{1-m} \cos(m\theta)+\dots$.  Then $a_1^{(m)}=2^{1-m}\chi^m(-T+x\cdot\omega)X\alpha_m$ for $t=T$; therefore we can recover $X\alpha_m$ by Proposition~\ref{pr_Xg_even}, and then $\alpha_m$.
\end{proof}


\section{Numerical examples} \label{sec-num}
In the examples below, we work in the square $[-1,1]^2$ discretized to $N\times N$ nodes with $N$ at least $1,000$, with   $h=0.005$. Then the wavelength is $2\pi h=0.0314\dots$ which is about $15.7$ times larger than the step size. We also take $N=2,000$ and even $N=4,000$ when capturing higher order harmonics is essential. We use a finite difference solver to compute the solution of \r{1} numerically. The terminal time is $T=1.4$. 

Figure~\ref{fig_setup} illustrates the setup. The probing wave $u_\textrm{in}^\mathbb{R}$ (or its real part if we take its complex version $u_\textrm{in}^\mathbb{C}$) at $t=0$ is plotted on the left, and the Cauchy condition for its $t$-derivative is chosen so that it would propagate up. We take non-linearities of the type $f(x,u)= \alpha(x)  F(u)$ (this particular form is not essential for the computations) with $\alpha$ a Gaussian plotted in the middle of the figure. Finally, the plot on the right is that of $\Lambda(u_\textrm{in}^\mathbb{R})$ (or its real part), i.e., the solution at $t=T=1.4$. Since the effect of the non-linearity is in the subprincipal term, it has no visible effect on that plot. Instead, in some of the figures below, we plot $u-u_L$, where $u_L$ is the linear solution; i.e., we subtract the principal term. Then we plot a vertical cross-section of $u-u_L$ through the center in the direction shown: from top to bottom. We plot the top $1/3$ (approximately) of the cross-section, where $u-u_L$ is essentially supported. 

\begin{figure}[ht]
\begin{center}
	\includegraphics[trim = 0 22 0 0 , scale=0.22]{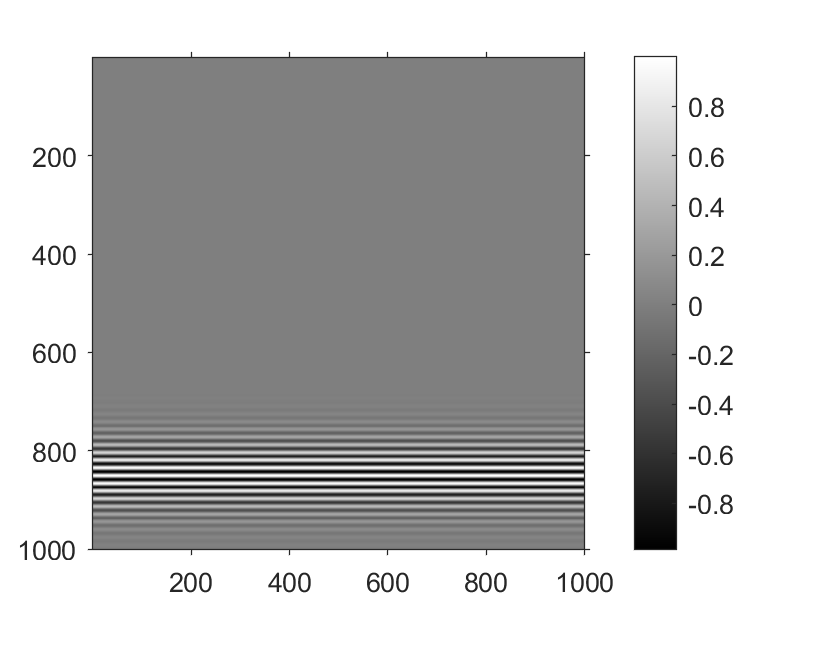}
	\includegraphics[ trim = 0 22 0 0 , scale=0.22]{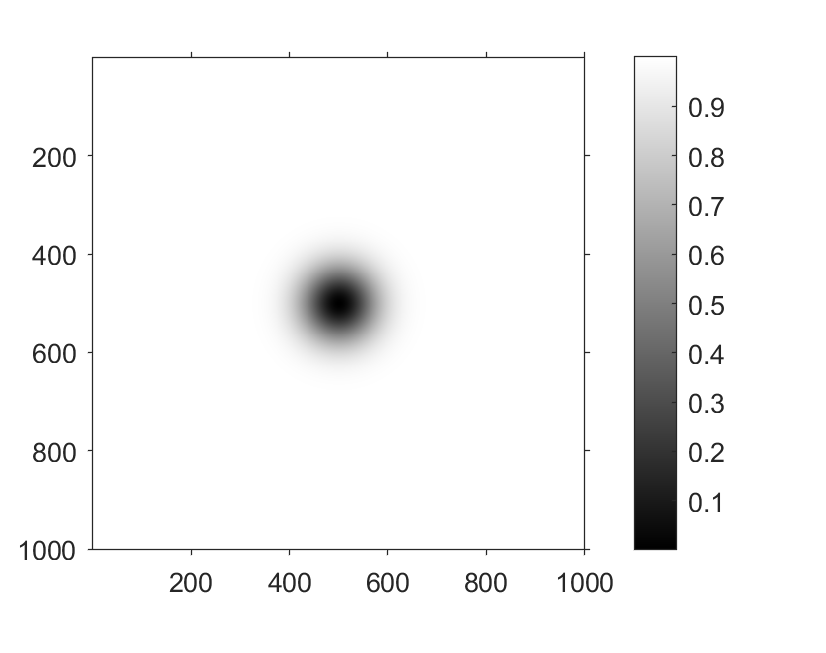}
	\includegraphics[trim = 0 22 0 0, scale=0.22]{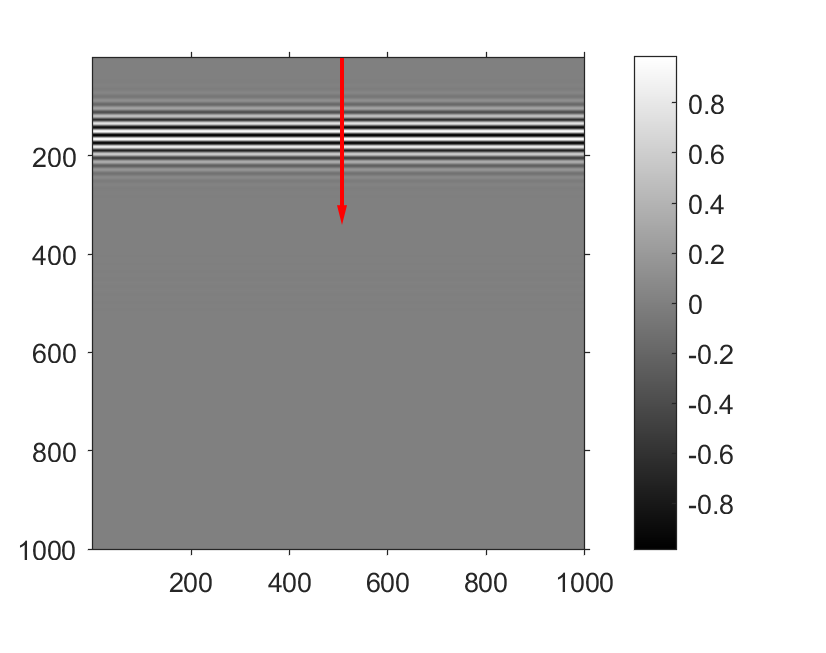}
\end{center}
\caption{\small The setup. All plots are of the real parts. 
Left: The   initial condition. Center:   the non-linearity $\alpha(x)$. Right: the solution at $t=T$. 
}
\label{fig_setup}
\end{figure}

\subsection{Cubic nonlinearity, complex incident wave} 
We start with a complex probing wave.
\begin{example} [Figure~\ref{fig_cubic_complex}]
 We take the nonlinearity to be $f=\alpha(x)u^3$ as in section~\ref{sec_cubic}. The subprincipal term in \r{main_eq} then is
\[
\frac{h}2 e^{\i(-T+x\cdot\omega)}\chi^3(-T+x_\parallel)X\alpha(x_\perp,\omega), \quad  \text{where\ } X\alpha(x_\perp,\omega)= \int \alpha(x_\perp+s\omega)\, \d s. 
\]

\begin{figure}[ht]
\begin{center}
	\includegraphics[trim = 0 0 0 0 , scale=0.22]{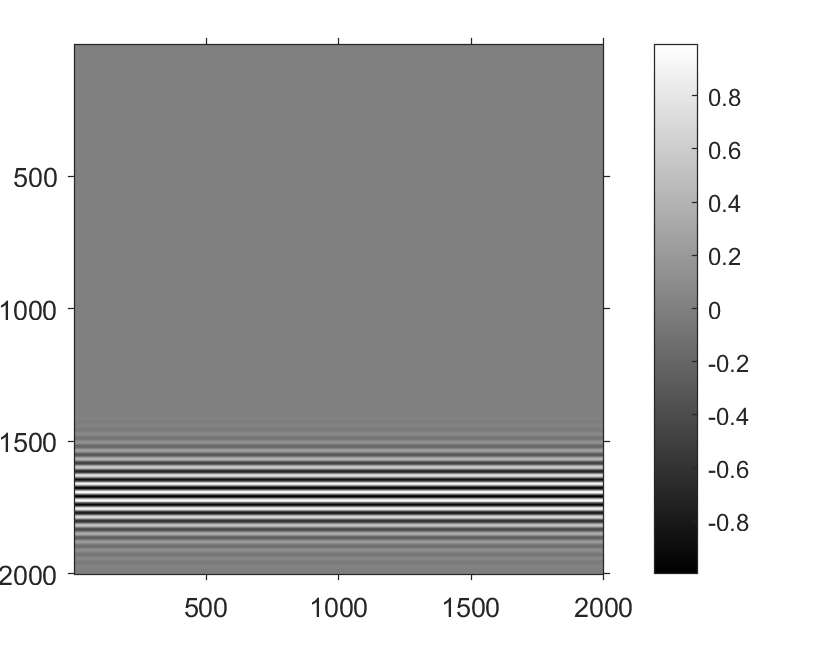}
	\includegraphics[ trim = 0 0 0 0 , scale=0.22]{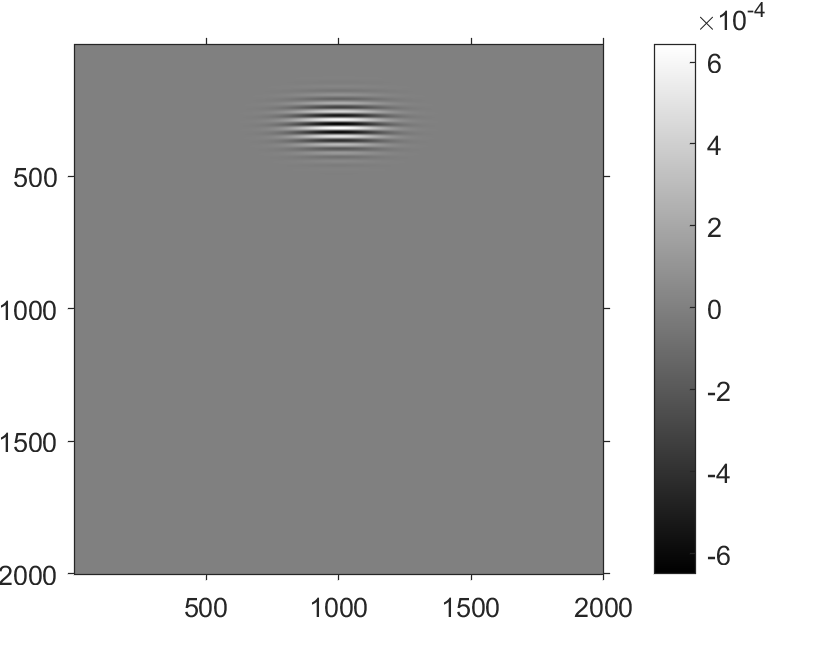}
	\includegraphics[trim = 0 0 0 0, scale=0.22]{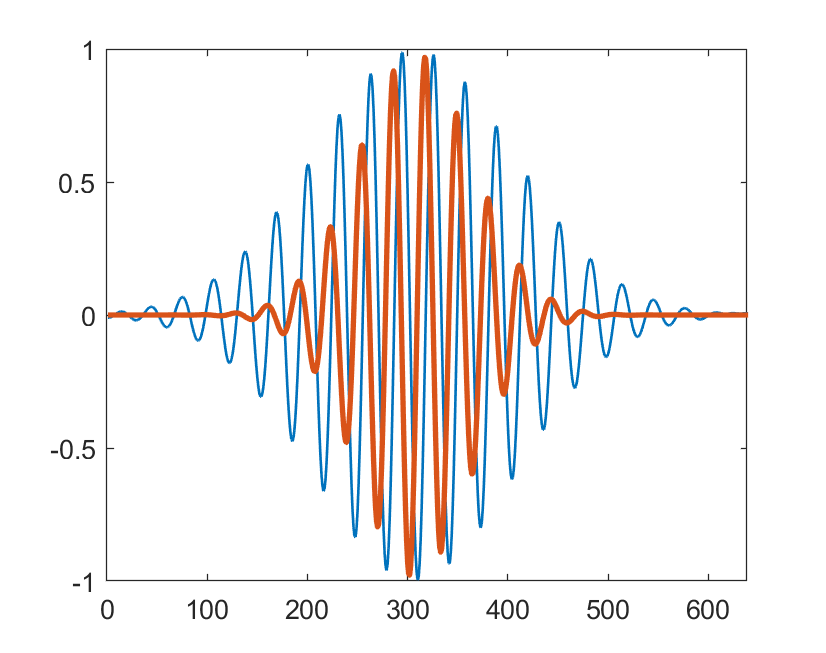}
\end{center}
\caption{\small $f(x,u)=\alpha(x)u^3$, $N=2,000$, complex $u_\text{in}$.  All plots are of the real parts. Left: The   initial condition. Center:   the subprincipal term at $t=1.4$. Right: Plot of a vertical cross section of the subprincipal term through the center (in red) and that of the linear solution (in blue). 
}
\label{fig_cubic_complex}
\end{figure}

In Figure~\ref{fig_cubic_complex}, we plot the real parts of the subprincipal term  (in red) vs.\ the principal one; with the subprincipal term rescaled to have the same maximal amplitude, i.e., divided by $h(X\alpha)/2$. We see that the subprincipal term decays faster away from the center because it has $\chi^3$ as an envelope of the oscillations instead of $\chi$. The two oscillations are shifted by $\pi/2$ which is due to the $\i=\sqrt{-1}$ factor. 
\end{example}

\subsection{More general non-linearities with separated variables, complex incident wave} \label{sec_num_sep} 
Assume $f(x,u) = \alpha(x) F_0(|u|^2)u$.   
Then \r{main_eq} takes the form
\be{main_eq_cubic_g}
\Lambda(u_\text{in})  = e^{\i \phi/h} \chi(\phi)+    \i \frac{h}2 e^{\i \phi/h}(X\alpha) \chi(\phi) F_0(\chi^2(\phi)) +O(h^2), \quad \phi:=-T+x\cdot\omega.
\ee
The envelope of the oscillations then is proportional to $(X\alpha) \chi(-T+x_\parallel) F_0(\chi^2(-T+x_\parallel))$.

\begin{example}[Figure~\ref{fig_cubic_complex_g}]
We choose $F(u) := F_0(|u|^2)u= 1.5\sqrt{3}(1-|u|^2)u$. It has a maximum $1$ on the interval $[0,1]$. Then we chose $\chi(s)=\exp(-s^2/0.02)$. A plot of the function $F\circ\chi= F_0(\chi^2(\cdot)) \chi(\cdot)$ and $-F\circ\chi$ are shown on Figure~\ref{fig_cubic_complex_g}, left. By \r{main_eq_cubic_g}, it must be the envelope of the oscillations of the subprincipal term and the computed profile in Figure~\ref{fig_cubic_complex_g}, right, confirms that. The curve in red is the subprincipal term divided by $h(X\alpha)/2$ as above.  

\begin{figure}[ht]
\begin{center}
	\includegraphics[ trim = 0 0 0 0 , scale=0.22]{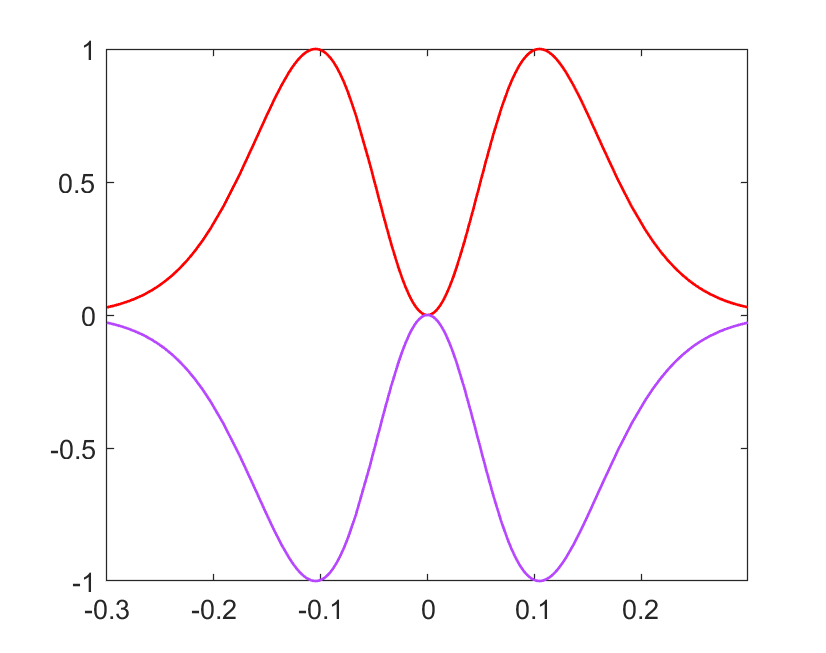} \hspace{20pt} 
	\includegraphics[trim = 0 0 0 0, scale=0.22]{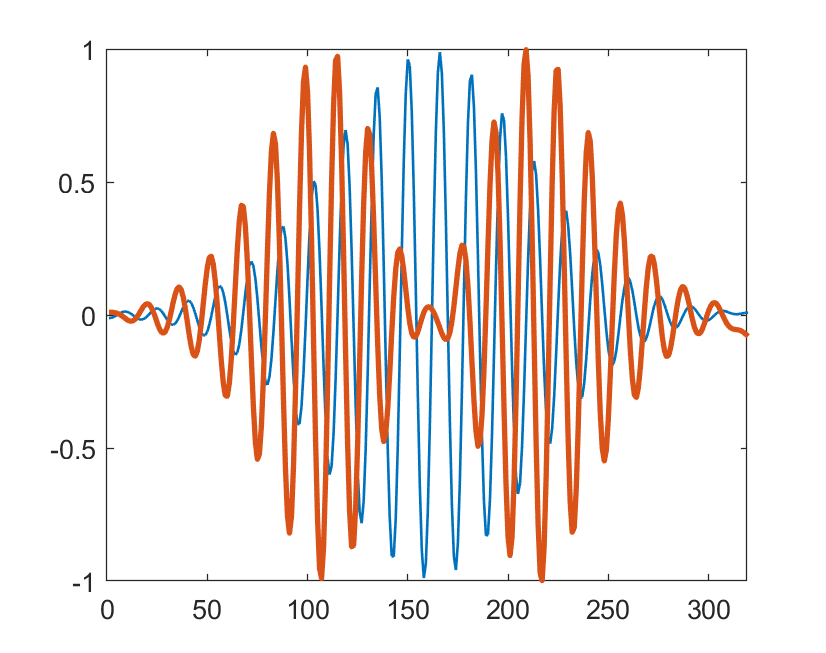}
\end{center}
\caption{\small  $f(x,u)=\alpha(x) 1.5\sqrt{3}(1-|u|^2)u$, $N=1,000$, complex $u_\text{in}$.  All plots are of the real parts. Left: the   theoretical envelope $F\circ \chi$ of the oscillations.  Right: a plot of a vertical cross section of the subprincipal term through the center, rescaled (in red) and that of the linear solution (in blue). 
}
\label{fig_cubic_complex_g}
\end{figure}

The recovery of $\g$ then goes along the following lines. We fix $x_\parallel$, which fixes $M:= \chi(-T+x_\parallel)$. Then the subprincipal term in \r{main_eq_cubic_g} recovers $F(M) X\alpha(x_\perp,\omega)$. Knowing this for all $\omega$ recovers $F(M)\alpha(x)$. Now, varying $x_\parallel$, we vary $M$ on the range of $\chi$, which recovers $F(M)\alpha(x)$. It is also worth noticing that the envelope of the oscillations in Figure~\ref{fig_cubic_complex_g}, right, is proportional to $F\circ\chi$, which recovers $F$, and therefore, $f$, up to a constant factor even without inverting the X-ray transform. 
\end{example}

\subsection{Non-linearities with separated variables, a real incident wave} 
We assume a non\-linearity as in section~\ref{sec_num_sep}. Then $f(x,u)= \alpha(x) F(u)$ by definition, where $F(u) = f_0(|u|^2)u$. We have
\be{N2}
\gamma_m(x,M) = \alpha(x)\tilde \gamma_m(M), \quad \tilde \gamma_m(M):=\frac4{M\pi}  \int_{0}^1 \left[F(Mq ) \right] T_m(q) \frac{\d q} {\sqrt{1-q^2}} . 
\ee
see \r{R3}. By \r{R8a}, \r{R11b},
\be{N3}
X\gamma_k = (X\alpha) \tilde\gamma_m(-T+x_\parallel).
\ee
Then the subprincipal term in \r{R8} takes the form
\be{N4}
h X\alpha(x_\perp) \chi(-T+x_\parallel)\sum_{k\ge1, \, \textrm {odd}}  \frac1{2k} \sin\frac{k(-T+x_\parallel)}{h} \tilde\gamma_k(\chi(-T+x_\parallel)).
\ee


Assume now that $F$ is a polynomial. Then $F(u)=\sum_{m\ge1,\,\textrm{odd} } F_m u^m$. There is a well known expansions of monomials $q^m$ into Chebyshev polynomials; and we can use that in \r{N2}. We will do this for polynomials $F$ of degree five only. We have
\[
q=T_1(q), \quad q^3 = \frac14 T_3(q) +\frac34 T_1(q), \quad q^5 = \frac1{16} T_5(q) +\frac{5}{16}T_3(q) +\frac58 T_1(q).
\]
Then 
\[
\tilde \gamma_1 = F_1 +\frac34F_3M^2 +\frac58 F_5 M^4, \quad \tilde \gamma_3 = \frac14  F_3 M^2+ \frac5{16}F_5 M^4  , \quad \tilde\gamma_5 = \frac1{16} F_5M^4 .
\]
The subprincipal term then takes the form, see \r{N4},
\be{N4a}
\begin{split}
h(X\alpha) &M\bigg[ \frac12 \left( F_1 +\frac34F_3M^2 +\frac58 F_5 M^4\right) \sin\frac{\phi}h \\
&  \quad +\frac16\left(     \frac14  F_3 M^2+ \frac5{16}F_5 M^4   \right) \sin\frac{3\phi}h  + \frac1{10} \cdot\frac1{16} F_5M^4\sin\frac{5\phi}h
\bigg]
\end{split}
\ee
In that case, we can easily demonstrate the two algorithms for recovery of the non-linearity. To make things simple, assume that $\alpha$ is known. If we know \r{N4a} \textit{for a all} $M$ in some interval $[0,K]$, then we can easily recover $F_1$, $F_2$ and $F_3$ from the first harmonic in \r{N4a}, which is proportional to $\tilde\gamma_1$. On the other hand, if $M$ is \textit{fixed}, we can recover the three $\tilde\gamma_k$, and then recover $F(Mq)$ by its Chebyshev expansion.

\begin{example}[cubic nonlinearity, Figure~\ref{fig_cubic_real}]
In particular, if $F(u)=|u|^2u$ (equal to $u^3$ for $u$ real), we get $\tilde \gamma_1=3M^2/4$, $\tilde \gamma_3=M^2/4$, all other zero, see \r{Rp0}; therefore, the subprincipal term is 
\be{N5}
h X\alpha(x_\perp) \chi^3(-T+x_\parallel) \bigg( \frac34 \sin\frac{-T+x_\parallel}{h} + \frac1{24} \sin\frac{3(-T+x_\parallel)}{h}\bigg).
\ee

In Figure~\ref{fig_cubic_real}, left, we plot a vertical cross-section of $u_{t=T}$ and that of the difference $(u-u_L)/(h X\alpha)$ at $t=T$. The theoretical maximum of that difference would be $3/4+1/24\approx 0.79$ (modulo $O(h)$ because it would depend on the phase shift of the oscillations in \r{N4} related to $x_\parallel=T$ maximizing $\chi$ there). This is in good agreement with the experimental plot. 

Next to that plot, in Figure~\ref{fig_cubic_real}, right, we plot the modulus of the Fourier transform of that difference on a log scale. We can see a peak at the incident frequency: the first (the carrier) harmonic, the third harmonic affecting the subprincipal term (and the lower order ones) and the fifth one affecting the lower order terms starting from the sub-subprincipal one. 
\begin{figure}[ht]
\begin{center}
\includegraphics[trim = 0 0 0 0, scale=0.22]{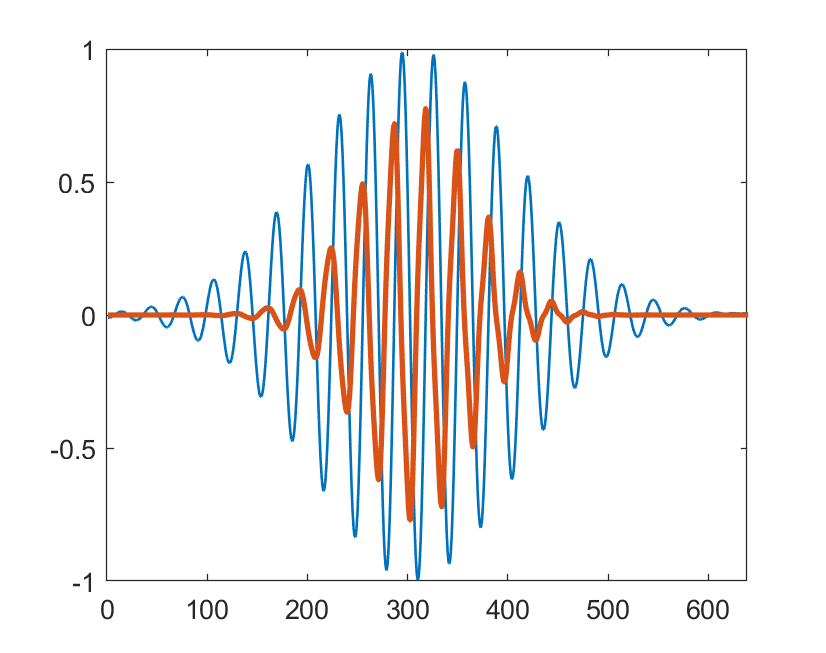}\hspace{20pt}
	\includegraphics[ trim = 0 0 0 0 , scale=0.22]{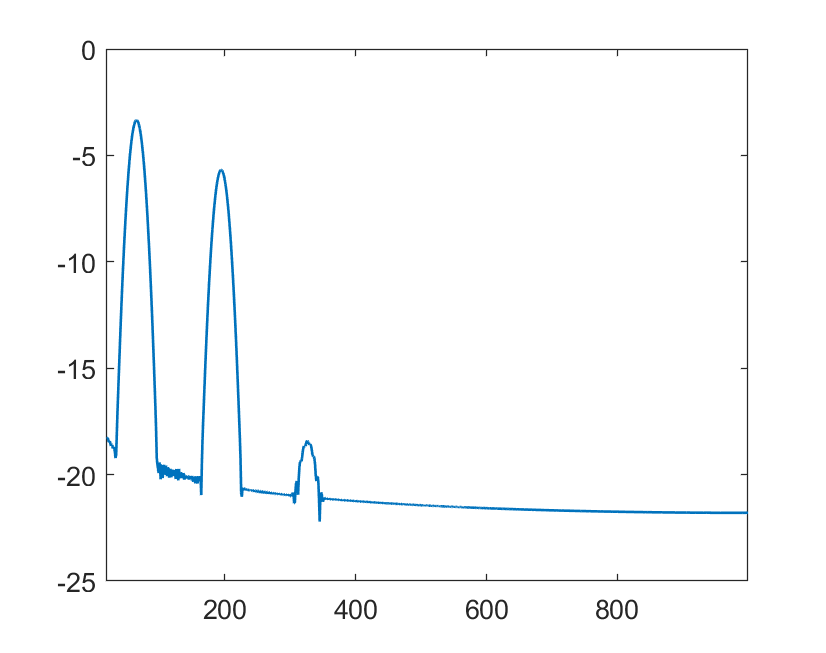}  
\end{center}
\caption{\small $f(x,u)=\alpha(x) u^3$,  $N=2,000$, real $u_\text{in}$. Left: a plot of a vertical cross section of $u-u_L$ (scaled) through the center (in red) and that of $u$ (in blue). Right: the  modulus of the Fourier transform on a log scale of it: the first, third harmonics dominate, and the fifth  one is coming from the third order term. 
}
\label{fig_cubic_real}
\end{figure}
\end{example}

\begin{example}[Figure~\ref{fig_pol_real}] 
In the next example, we choose a polynomial of degree three in order to make the third harmonics not just measurable but also visible. We choose $F_1=1$, $F_3=-1.9$, i.e.,
\[
F(u) = u-1.9u^3.
\]
The first term represents a linear potential $V=1$, actually. 
The coefficients are chosen so that $\tilde \gamma_1=0$ for $M=0.84$, which would kill the first harmonic when $\chi=0.84$ (close to its peak $\chi=1$). We plot the theoretical profile of the subprincipal term divided by $h$,  versus the computed one in Figure~\ref{fig_pol_real}. 
\begin{figure}[ht]
\begin{center}
\includegraphics[trim = 0 0 0 0, scale=0.22]{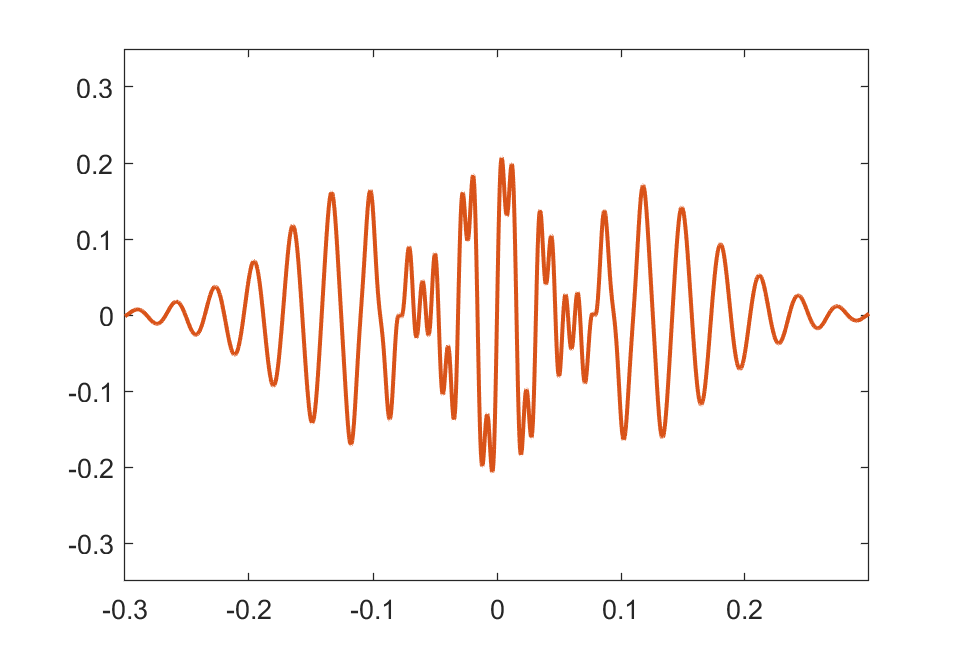}\hspace{20pt}
	\includegraphics[ trim = 0 0 0 0 , scale=0.22]{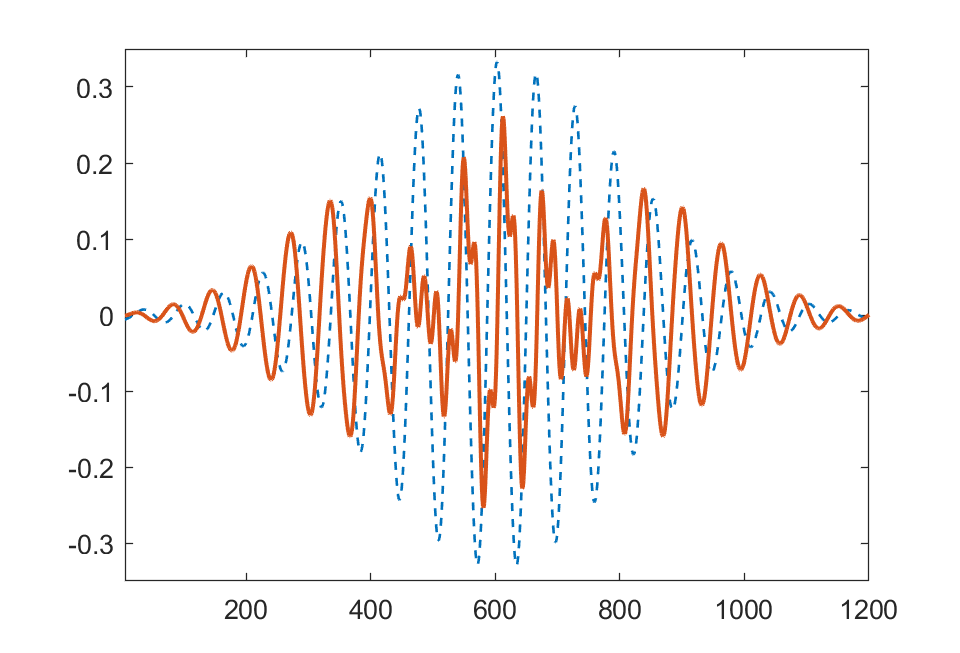}  
\end{center}
\caption{\small   $f(x,u)=\alpha(x)(u-1.9u^3)$, $N=4,000$, real $u_\text{in}$. Left: a plot of the theoretical vertical cross section of $(u-u_L)/h$. Right: the computed one vs.\ $0.33u_L$. 
}
\label{fig_pol_real}
\end{figure}
There is no perfect match because of numerical dispersion: higher frequencies travel slower. Still, one can see evidence of the third harmonics near the center part. The effect is much stronger and eclipses the first harmonic approximately  where $\chi$ is closer to $0.84$. The linear solution is plotted divided by $3$.
\end{example}

\begin{example}[quadratic non-linearity, Figures~\ref{fig_zero}, \ref{fig_zero2}]
Let $f(x,u)=\alpha(x)u^2$. We recall the analysis in section~\ref{sec_q2}. 
We compute a linear solution  $u_L$ first, i.e., with $\alpha=0$. It is equal to $u_\textrm{in}^\mathbb{R}$, of course, which also represents the first term in \r{E01}. Then we compute $u_{0,\textrm{theor}}^{(0)}$ as a numerical solution of \r{E03}. This would be the theoretical zeroth mode, up to a lower order term. Next, we compute $u$ itself. The zeroth mode $u_0^{(0)} $ is relatively small since its size depends on $\chi$ as well (which has small ``support'' in our case), and it propagates in all directions; therefore it spreads. For this reason, we do not plot $u|_{t=T}$; it looks more or less as the unperturbed $u_L|_{t=T}$ in this case. In Figure~\ref{fig_zero}, we plot $u-u_L$, $u_{0,\textrm{theor}}^{(0)}$, and $u-u_L-u_{0,\textrm{theor}}^{(0)}$   at $t=T$. According to \r{E01}, we should have
\[
\begin{split}
u-u_L &= u_{0,\textrm{theor}}^{(0)}+ O(h),\\
u-u_L-u_{0,\textrm{theor}}^{(0)} &= h\left( u_1^{(0)} -2 (\Im a_1^{(1)})\sin(\phi/h) -2 (\Im a_1^{(2)})\sin(2\phi/h)\right) +O(h^2).
\end{split}
\]
Therefore, the first two plots should be very close. Indeed, they are, and in $u-u_L$, we see a slight hint of the oscillations due to the subprincipal term.
\begin{figure}[ht]
\begin{center}
	\includegraphics[trim = 0 0 0 0 , scale=0.22]{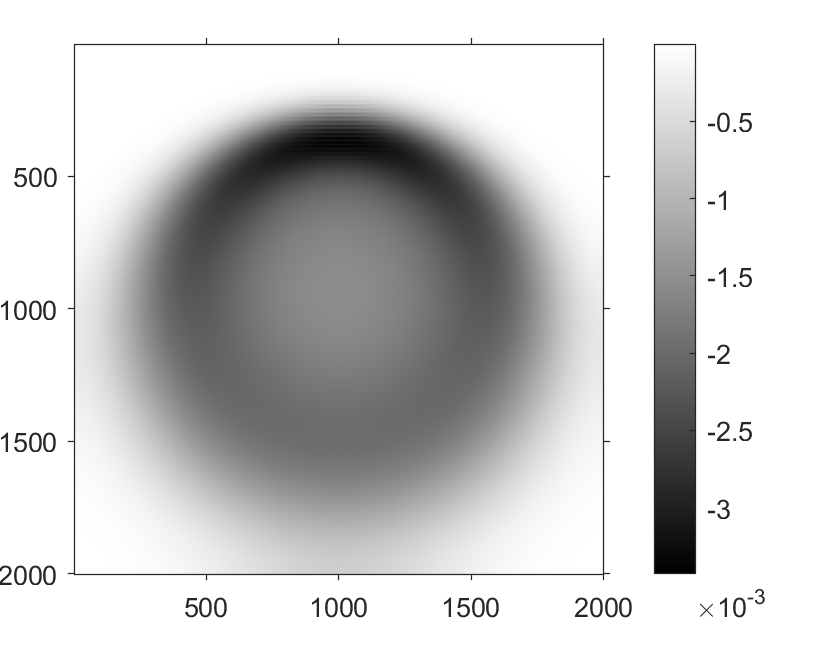}
	\includegraphics[ trim = 0 0 0 0 , scale=0.22]{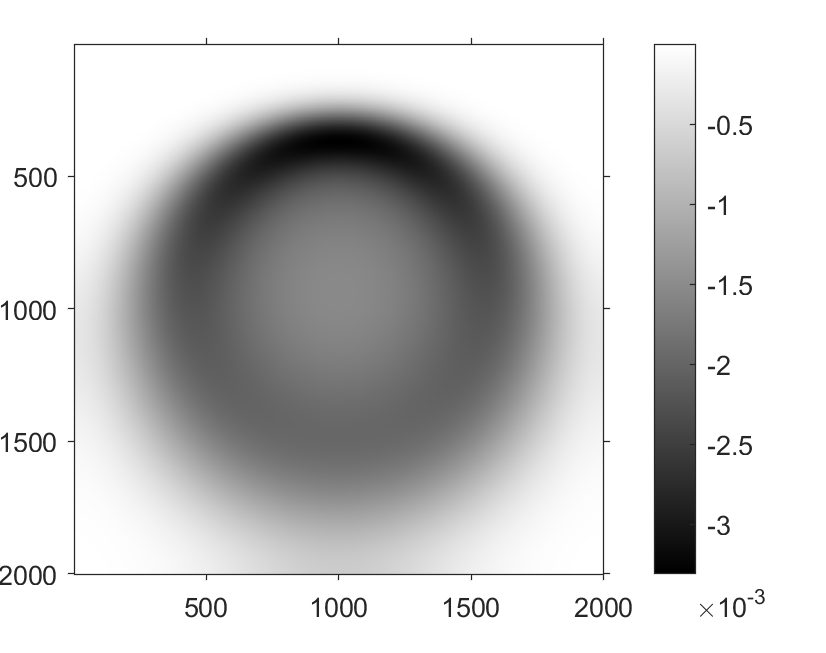}
	\includegraphics[trim = 0 0 0 0, scale=0.22]{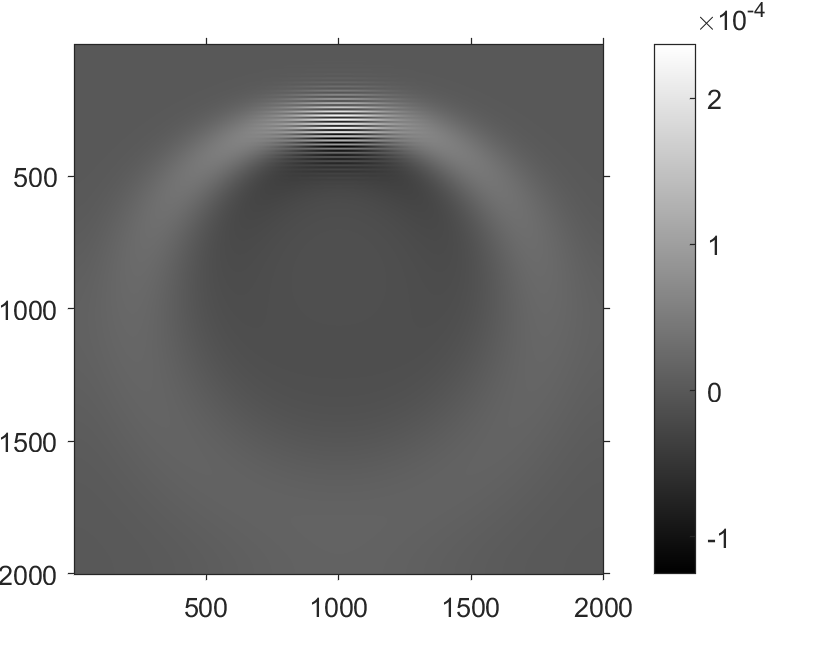}
\end{center}
\caption{\small $f(x,u) = \alpha(x)u^2$, $N=2,000$,   $t=T=1.4$.  
Left: $u-u_L$,  represents the computed zeroth harmonic up to $O(h)$. Center:   $u_{0,\textrm{theor}}^{(0)}$, the theoretical zeroth harmonic up to $O(h)$. Right: The difference of the two, representing the subprincipal term up to $O(h^2)$. 
}
\label{fig_zero}
\end{figure}
The third plot represents the difference $u-u_L-u_{0,\textrm{theor}}^{(0)}$, and we see the oscillations and the contribution of the subprincipal zeroth order term $u_1^{(0)}$. Note that the third plot is on a different scale --- about 15 times smaller. 

In Figure~\ref{fig_zero2}, we plot vertical profiles of $(u-u_L)/h$, left, and $(u-u_L- u_{0,\textrm{theor}}^{(0)})/h$, right. Even though $u_0^{(0)}$ is a principal order term, as mentioned above, it is much smaller than the leading cosine term but still larger than the harmonic frequencies in the subprincipal term. The red curve on the left is an approximation of $h^{-1} u_0^{(0)}+ u_1^{(0)}$ plus the subprincipal oscillatory terms. One can see the second harmonic mostly but  $u_0^{(0)}$ dominates. With that low frequency wave subtracted (computed as a solution of \r{E03}), we see, on the right, the second harmonic mostly plus the subprincipal term of the zeroth wave $u_1^{(0)}$. Numerical calculations, not shown, using the formula \r{E04} to compute the theoretical amplitude of the second harmonic, agree very well with the plot in Figure~\ref{fig_zero2}, right. 

\begin{figure}[ht]
\begin{center}
	\includegraphics[trim = 0 0 0 0 , scale=0.25]{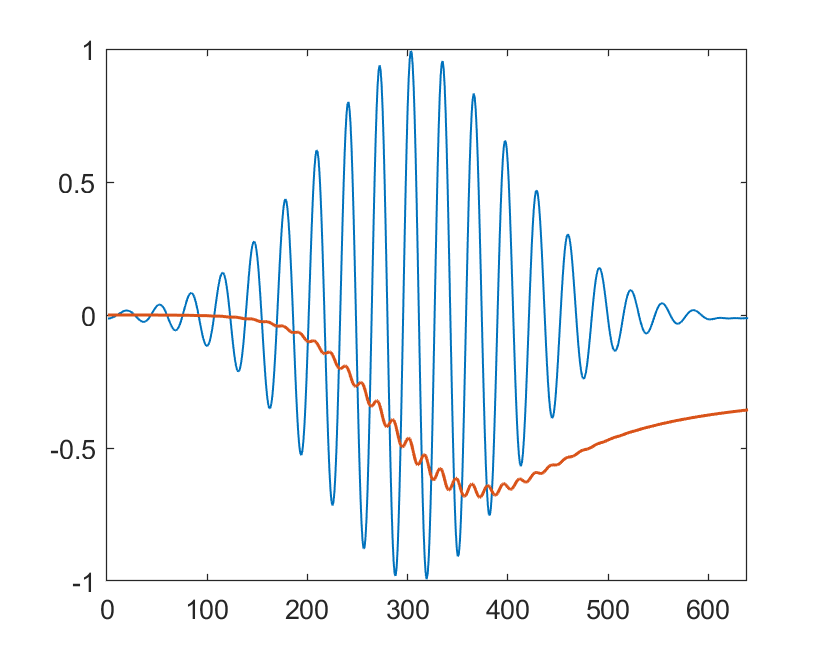}
	\includegraphics[ trim = 0 0 0 0 , scale=0.25]{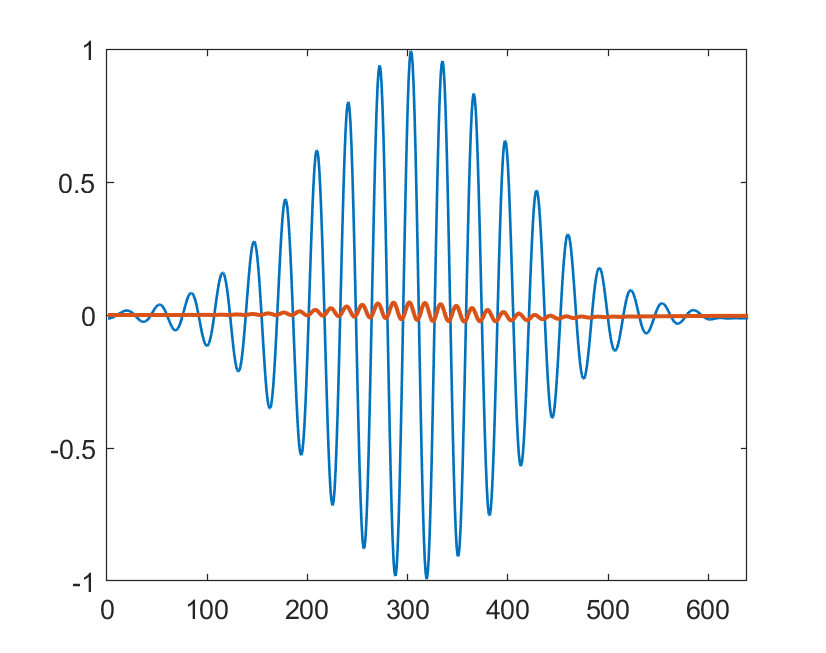}
\end{center}
\caption{\small As in Figure~\ref{fig_zero}. Vertical profiles of $(u-u_L)/h$, left, and $(u-u_L- u_{0,\textrm{theor}}^{(0)})/h$, right; both plotted against $u$ (in blue). 
}
\label{fig_zero2}
\end{figure}

\end{example}
 
\section{Remarks} \label{sec_remark} 
In our main results, we do not impose assumptions on $f$ guaranteeing existence and uniqueness of solutions of \r{1} for $t\in[0,T]$ with arbitrary, even smooth but not small data. We prove that \r{1} is solvable with the initial conditions \r{10a}, or its real part, for $h\ll1$; and that the solution is smooth. By Theorem~\ref{unique-1}, this solution is unique in the class of the bounded solutions at least (then the bound $\|u_j\|_{L^\infty}\le M$ in Theorem~\ref{unique-1} holds). We do not exclude a priori existence of other, unbounded solutions. If $f$ satisfies the assumptions of Theorem~\ref{thm_K} however, such solutions do not exist.

 Assuming that $f=f(t,x,u)$ depends on $t$ as well poses no problems in the asymptotic analysis. In the odd case, we recover the light-ray transform of $\g(t,x,p)$ instead of the X-ray one of $\g(x,p)$, and similarly for real probing fields and $f(t,x,u)=\alpha(t,x)u^{2m}$. 
 We would need to involve another one-dimensional ``delay'' parameter in the incident waves however.  The light ray transform, known even  over a bounded time interval recovers $\g$ (for any fixed $p$) uniquely in the domain covered by those rays \cite{S-support2014}. 
 The methods extend naturally to a wave equation of the type $u_{tt}-\Delta_gu=0$ related to a Riemannian metric or even to $\Delta_g$ involving a magnetic and an electric potential, as long as there are no caustics. 
  In fact, an electric potential term $Vu$ can be absorbed in the non-linearity. Even Lorentzian wave operators can be treated. On the other hand, the inversion of the light ray transform on non-Minkowski manifolds is a more delicate problem. 

One can pose the problem as a boundary value problem, in principle, on a cylinder $[0,T]\times \bo$, given a certain (non-linear) outgoing Dirichlet-to-Neumann  map; where $\Omega\subset\R^n$ is a fixed domain. Under the assumption that the $x$-support of the non-linearity is contained in $\Omega$,  a  reduction from our formulation to such one and back is trivial. If we do not have the support assumption, we would have to deal with the boundary value problem for \r{1} first, a task we would like to avoid for the purpose of this work. On the other hand, in authors' view, the  Dirichlet-to-Neumann  map formulation for the wave equation is a bit artificial, while probing the medium with waves coming from outside at $t=0$ and measuring them outside again, for $t=T$ is a more natural formulation. 

The compactness of $\supp_x f$ does not seem to be needed. It certainly simplifies the exposition though. 

One can use probing waves with  support localized near a point of the type $u_\textrm{in}  = e^{\i(-t+x\cdot\omega)}\psi(x)$ (or its real part), $\psi\in C_0^\infty$, at some initial moment, and an appropriate initial condition for $u_t$, compare with \r{10a} or \r{10real}. It would be simpler to  multiply \r{10a} or \r{10real} by a $C_0^\infty$ function of $x_\perp$. Note that this localization would be $h$-independent, this would add more terms to the asymptotic expansion but it will not change the main idea. 

We restrict the space dimension to $n=2$ or $n=3$ because the stability estimate \r{stability2} in Theorem~\ref{stability} has a short proof then. 

\section{Discussion} \label{sec_discussion}
There has been an increased interest in inverse problems for non-linear hyperbolic equations recently. The pioneering works \cite{KLU-18} and \cite{LassasUW_2016} suggested the higher order linearization idea mentioned in the introduction: send a wave $u_0=\eps_1 u_1+\eps_2 u_2 +\eps_3 u_3 +\eps_4 u_4$, where $u_j$ are chosen so that they carry conormal singularities colliding at a chosen in advance point in time-space. This collision creates a point source through non-linear interaction which emits a spherical wave in the $\eps_1\eps_2 \eps_3\eps_4$ term as all $\eps_j\to0$. This term must be separated from all other, and its (weak) singularity has to be measured to recover the non-linearity at that point in time-space. This method has been used in other papers, see, e.g., \cite{Hintz1,Hintz2,LassasUW_2016, LUW1, lassas2020uniqueness, Hintz-U-19}. In \cite{OSSU-principal}, three such waves are sent and each one is a (linear) Gaussian beam. 

In \cite{S-Antonio-nonlinear}, we proposed a different direction: send a wave which propagates in the so-called weakly non-linear regime. It has an amplitude $h^{-1/2}$ and a wavelength $\sim h$. The math theory of the weakly non-linear propagation has been developed in  \cite{Metivier-Notes,Metivier-Joly-Rauch, Joly-Rauch_just, Donnat-Rauch_dispersive, Dumas_Nonlinear-Geom-Optics, JMR-95, Rauch-geometric-optics} (and other works), and it is known in the physics literature as well. Most of those work are on fist order non-linear systems. 
 The waves there have the asymptotic behavior $u\sim h^p U(\phi/h,t,x)$, where $p$ is chosen so that the non-linearity does not affect the eikonal equation for $\phi$ but affects the leading transport equation. In the equation $\Box u+\alpha |u|^{J-1}u=0$, for example, one gets $p=-1/(J-1)$; in particular $p=-1/2$ for a cubic non-linearity. The first transport equation is a non-linear ODE, or a system of such; and we showed that its solution induces a phase shift proportional to the X-ray transform of $\alpha$, which allows us to recover $\alpha$. The essential difference of the approach in \cite{S-Antonio-nonlinear} with the higher order linearization method is that we propagate waves in an actual non-linear regime; and the signal carrying the useful information is the principal one instead of a signal or order $-4$. 

We want to emphasize that none of the approaches above can recover general non-linearities $f(x,u)$ away from the two extremes $u\sim0$ and $u\sim\infty$. The first approach relies on asymptotically small waves having the potential to recover the Taylor expansion of $f(x,u)$ about $u=0$. In fact, this is what is done in \cite{LassasUW_2016}, starting with terms of order $4$ and higher. The second one requires an asymptotic expansion of $f$  w.r.t.\ $u$ when $u\to\infty$ (at least as a leading term), requires strong assumptions on $f$ for solvability, and one would expect that it has the potential to reliably recover that asymptotic behavior only, see also \cite{S-Antonio-nonlinear}.

This brings us to the main idea of this paper: use solutions for which $p=0$ above, i.e., $u\sim 1$. Then the non-linearity affects the subprincipal term but not the principal one of high-frequency solutions. We would expect to need high-frequency solutions in order to get good resolution. 

The geometric optics analysis in the papers cited above, see, e.g., \cite{Metivier-Notes} for a survey, when applied to the wave equation, is not restricted to polynomial non-linearities only (they are all $x$-independent though) but in effect, they always rely in the Taylor expansion of $f$ about $u=0$. In many of them, the small parameter $h$ (called there $\eps$) is a part of the PDE as well, for example the derivatives could be replaced by their semi-classical ones $h\p_t$, $h\p_x$, or similar. This is equivalent to rescaling time and/or space. This would change the value of $p$ above for which the propagation is in the weakly non-linear regime, and it could make it $p=0$. This happens for the PDE $h\Box u+\alpha |u|^{J-1}u=0$, for example, which is obtained from its $h$-independent version after the scaling $t'=h^{1/2}t$, $x'=h^{1/2} x$. We are studying $h$-independent PDEs however over $h$-independent domains in time-space. 

We also want to mention the method of linearization near a non-zero solution initiated by Isakov. In \cite{Kian-semilinear}, a uniqueness result is proven when there are also internal measurements at $t=T$ as well.

\appendix 
\section{Properties of Solutions of Semilinear Wave Equations}\label{sec_app}

We prove some results for semilinear wave equations which are known to experts, see \cite{Kap1,Kap2,ShStruwe2},  but their proofs do not seem to be readily available and so we include them here for completeness and because they are elementary. We should also mention that part of this discussion, including a different version of Theorem \ref{unique-1}, is contained in \cite[Ch.~6]{Hormander-nonlinear-book},

Let $g=\sum_{j,k=1}^n g_{jk}(x) \d x_i \d x_k$ be a $C^\infty$ Riemannian metric in $\mathbb{R}^n$, $n\geq 2$, 
such that $g_{jk}(x)=\delta_{jk}$ if $|x|\geq R$,  for some $R>0$,  and  let $\Delta_g$ be the corresponding (negative) Laplacian:
\[
\Delta_g=\frac{1}{\sqrt{|g(x)|}} \partial_{x_j} \sqrt{|g(x)|} g^{jk}(x) \partial_{x_k}, \;\  |g(x)|=\det (g_{jk}(x)), \text{ and } 
[g^{jk}(x)]=[g_{jk}(x)]^{-1}.
\]

We start by recalling an energy estimate for solutions of the linear wave equation:
If $u(t,x)$ satisfies
\[
\begin{split}
 u_{tt}-\Delta_g u &= G(t,x), \\
u(0,x)= \varphi(x) &,  \;\ u_t(0,x)=\psi(x),
\end{split}
\]
we have
\[
u(t) = \cos(t\sqrt{-\Delta_g})\varphi + t\sinc (t\sqrt{-\Delta_g})  \psi + \int_0^t (t-s)\sinc\big((t-s)\sqrt{-\Delta_g})\big)G(s)\, \d s,
\]
where $\sinc(s)=\sin(s)/s$, and we suppressed the dependence on $x$. Then for any $T>0$,  and $s\geq 0$, 
\be{Eest}
\begin{split}
& \sup_{[0,T]} \| u(t,\cdot)\|_{H^{s+1}(\mathbb{R}^n)}+ \sup_{[0,T]} \|\p_tu(t,\cdot)\|_{H^s(\mathbb{R}^n)} \leq  CT E_{s+1}(0) ,\\
&  \text{ where } 
E_{s+1}(0)=  \|\varphi\|_{H^{s+1}(\mathbb{R}^n)}+  \|\psi\|_{H^s(\mathbb{R}^n)}+ \int_0^T \|G(t,\cdot)\|_{H^s(\mathbb{R}^n)} \d t,
\end{split}
\ee
provided the right hand side is finite.

We recall that a particular case of Gronwall's inequality states that if $f$ and $g$ are continuous functions on $[0,T]$,  and if
\be{GRW}
\begin{split}
 f(t) \leq g(t) +C  \int_0^t f(s) \d s,  \text{ then  for  } t\in [0,T]  \;\   f(t) \leq C_1g(t).
\end{split}
\ee

Next we prove an uniqueness theorem for solutions of a semilinear wave equation and we remark that a different version of this result can be found in  \cite[Theorem~6.4.10]{Hormander-nonlinear-book}.

\begin{theorem}\label{unique-1} Let $h(t,x,u)$ be a $C^\infty$ function which is compactly supported in $x$ and let 
$r(t,x)\in L^1(\mathbb{R}; L^2(\mathbb{R}^n) )$.    If 
\[
u_j \in C([0,T];  H^{1}(\mathbb{R}^n)) \cap C^1([0,T];  L^2(\mathbb{R}^n)), \;\ j=1,2,
\]
are such that there exists a constant $M>0$ such that $\|u_j\|_{L^\infty([0,T] \times \mathbb{R}^n)}\leq M$,  and $u_j$,  $j=1,2$,  
are weak solutions of the equation
 \be{Weq0}
\begin{split}
 u_{tt}-\Delta_g u + h(t,x,u) =r(t,x), \\
u(0,x)= \varphi(x) ,  \;\ u_t(0,x)=\psi(x),
\end{split}
\ee
then $u_1=u_2$.  
\end{theorem}
\begin{proof}

 We apply \r{Eest} to the difference $u_1-u_2$, and in this case,
\[
\begin{split}
 \left(\p_t^2 -\Delta_g\right) (u_1-u_2) = h(t,x,u_2)-h(t,x,u_1), \\
u_1(0,x)-u_2(0,x)= 0 ,  \;\ \p_tu_1(0,x) -\p_tu_2(0,x)=0.
\end{split}
\]
Since $h(t,x,u)$ is $C^\infty$,  and  compactly supported in $x$,  then there exists a $C^\infty$ function $h_1(t,x,u, v)$,  which is compactly supported in $x$,  such that
$h(t,x,u)- h(t,x,v)= h_1(t,x,u,v)(u-v)$.   We deduce from \r{Eest} that for any $t\in [0,T]$,
\[
\begin{split}
&\mathcal{M}(t):=  \|u_1(t,\cdot)-u_2(t,\cdot)\|_{H^1(\mathbb{R}^n)}+ \|\p_tu_1(t,\cdot)- \p_t u_2(t,\cdot)\|_{L^2(\mathbb{R}^n)} \leq  \\
&CT\int_0^t \| h_1(s,\cdot,u_1,u_2)(u_1-u_2)\|_{L^2(\mathbb{R}^n)} \d s.
 \end{split}
 \]
Since $u_1$ and $u_2$ are bounded, and $h_1$ is compactly supported in $x$,  then 
$\|h_1(t,\cdot,u_1,u_2)\|_{L^\infty([0,T] \times \mathbb{R}^n)}\leq M_1$  and so
\be{mtce}
\mathcal{M}(t) \leq  CT M_1     \int_0^t \|(u_1(s,\cdot)-u_2(s,\cdot))\|_{L^2(\mathbb{R}^n)} \, \d t  \leq  CT M_1   \int_0^t \mathcal{M}(s) \, \d s 
 \ee

We use Gronwall's inequality \r{GRW}, and in this  particular case of \r{mtce} $g(t)=0$ and so $\mathcal{M}(t)=0$ and therefore $u_1=u_2$. 
 \end{proof}

Next we study solutions of the semilinear equation
 \be{Weq2}
\begin{split}
 u_{tt}-\Delta_g u &+ \, \kappa(|u|^2)h(t,x,u) u= r(t,x), \\
u(0,x)= & \varphi(x),  \;\ u_t(0,x)=\psi(x),
\end{split}
\ee
which is relevant for the discussion above.  Here $r(t,x)\in L^1(\mathbb{R}; H^s(\mathbb{R}^n) )$, $h(t,x,u)\in C^\infty$ is compactly supported in $x$ and   $\kappa\in C_0^\infty(\mathbb{R})$, $\kappa(s)=0$  if $|s|>2\rho$ and $\kappa(s)=1$  if $|s|<\rho$.
 Our result for this case is the following:

\begin{theorem}\label{cut-off} Suppose $h(t,x,u)\in C^\infty$ is compactly supported in $x$.  Then for any non-negative integer $s$, and initial data $(\varphi,\psi)$ with $\varphi \in H^{s+1}(\mathbb{R}^n)$ and 
$\psi \in H^{s}(\mathbb{R}^n)$,  there exists a unique $u(t,x)$ such that $u\in C([0,T];  H^{s+1}(\mathbb{R}^n))$ and 
$\p_tu \in C([0,T];  H^{s}(\mathbb{R}^n))$ which satisfies \r{Weq2} (as a weak solution if $s$ is not large enough)
and moreover for any $T>0$, there exists a constant $K_s=K_s(T,s,\rho)>0$ such that
\be{estimate1}
\begin{split}
& \sup_{[0,T]} \| u(t,\cdot)\|_{H^{s+1}(\mathbb{R}^n)}+ \sup_{[0,T]} \|\p_tu(t,\cdot)\|_{H^{s}(\mathbb{R}^n)} \leq K_s E_{s+1}, \\
& \text{ where } E_{s+1}= \|\varphi\|_{H^{s+1}(\mathbb{R}^n)}+  \|\psi\|_{H^{s}(\mathbb{R}^n)} +  \int_0^T \|r(t,\cdot)\|_{H^s(\mathbb{R}^n)}  \, \d t.
\end{split}
\ee
\end{theorem}

\begin{proof}  Let us first consider the case $s=0$. Let  us denote $G(t,x,u)= \kappa(|u|^2) h(t,x,u) u$. Let $u_0$ be the solution of
\be{Weq3}
\begin{split}
 \p_t^2 u_{0}-\Delta_g u_0& =r(t,x), \\
u_0(0,x)= \varphi(x) &,  \;\ \p_t u_{0}(0,x)=\psi(x),
\end{split}
\ee
let $u_1$ satisfy
\be{Weq4A}
\begin{split}
 \p_t^2u_{1}-\Delta_g u_1 &+ G(t,x,u_{0})=0, \\
u_1(0,x)= 0 &,  \;\ \p_t u_1(0,x)=0,
\end{split}
\ee
and for $j>1$, let $u_j$ satisfy

\be{Weq4}
\begin{split}
\p_t^2 u_{j}-\Delta_g u_j &+ G(t,x, u_0+u_{j-1})=0, \\
u_j(0,x)= 0 &,  \;\ \p_t u_j (0,x)=0.
\end{split}
\ee

In view of \r{Eest}, and the fact that there exists $C>0$ such that $|G(t,x,u)|\leq C$, we find that
\[
\begin{split}
&\sup_{[0,T]} \| u_0(t,\cdot)  \|_{H^{1}(\mathbb{R}^n)}  + \sup_{[0,T]} \|\p_tu_0(t,\cdot)\|_{L^2(\mathbb{R}^n)}  \leq 
CT\big(  \|\varphi\|_{H^{1}(\mathbb{R}^n)}+  \|\psi\|_{L^2(\mathbb{R}^n)} + \int_0^T \|r(t,\cdot)\|_{L^2(\mathbb{R}^n)}  \, \d t\big), \\
 & \sup_{[0,T]} \| u_1(t,\cdot)\|_{H^{1}(\mathbb{R}^n)}+ \sup_{[0,T]} \|\p_tu_1(t,\cdot)\|_{L^2(\mathbb{R}^n)}  \leq 
CT \int_0^T \|G(t,\cdot,u_0(t,\cdot))\|_{L^2(\mathbb{R}^n)}\, \d t \leq CT  , \\
 & \sup_{[0,T]} \| u_j(t,\cdot)\|_{H^{1}(\mathbb{R}^n)}  + \sup_{[0,T]} \|\p_tu_j(t,\cdot)\|_{L^2(\mathbb{R}^n)}  \leq 
CT \int_0^T \|G(t,\cdot, u_0(t,\cdot)+ u_{j-1}(t,\cdot))\|_{L^2(\mathbb{R}^n)} \d t \le CT  .
  \end{split}
  \]

On the other hand, since $h(t,x,u)$ is $C^\infty$, and compactly supported in $x$, it follows that there exists a $C^\infty$ bounded function $\mathcal{G}$ such that
\[
G(t,x,u)- G(t,x,v)= \mathcal{G}(t,x,u,v) (u-v).
\]

Therefore, we have 
\[
\begin{split}
\p_t^2(u_j-u_k)- & \Delta_h( u_j-u_k) =  G(t,x, u_0+u_{j-1})- G(x, u_0+ u_{k-1})= \\
& \mathcal{G}(t,x, u_0+u_{j-1}, u_0+ u_{k-1})(u_{j-1}-u_{k-1}). 
\end{split}
\]
Therefore, again by standard energy estimates,  we have for $j,k\geq 1$, 
\[
\begin{split}
 &\sup_{[0,T]} \|u_j(t,\cdot)- u_k(t,\cdot)\|_{H^1(\mathbb{R}^n)}+ \sup_{[0,T]} \|\p_t u_{j}(t,\cdot)- \p_tu_{k}(t,\cdot)\|_{L^2(\mathbb{R}^n)} \leq  \\
 C T\int_0^T \|\mathcal{G}(t,x, &  u_0+  u_{j-1}, u_0+ u_{k-1})(u_j-u_k)\|_{L^2(\mathbb{R}^n)} \d t \leq 
 CT^2    \sup_{[0,T]} \|u_{j-1}(t,\cdot)- u_{k-1}(t,\cdot)\|_{H^1(\mathbb{R}^n)},
\end{split}
\]
with $C$ depends only on the $L^\infty$ norm of $\mathcal{G}(t,x,u,v)$ for $t\in [0,T]$,  which is of course independent of $u$ and $v$,  and so is independent of $j,k$. This shows that if $T_1$ is such that $CT_1^2\leq \frac12$, the sequence $\{u_j+u_0\}$ is Cauchy in the space $C([0,T_1]; H^1(\mathbb{R}^n) \cap C^1([0,T_1]; L^2(\mathbb{R}^n)$,  and therefore converges.   Notice that $u$ is bounded on the support of $\kappa(|u|^2)$,  so if $\eta\in C_0^\infty$, 
\[
\langle G(t,x,u),\eta\rangle= \int_{\{(t,x): |u|^2 \leq 3\rho, 0\leq t \leq T_1\}} G(t,x,u) \ \eta(t,x) \, \d t\, \d x.
\]
Also,
\[
\left|\int( G(t,x,u)-G(t,x,u_j+u_0)) \eta(t,x)\, \d t\, \d x\right| \leq C \int |u- u_j-u_0| |\eta| \, \d t\, \d x \rightarrow 0 \text{ as } j \rightarrow \infty.
\]
 Therefore, the  limit satisfies \eqref{Weq2} weakly, and this proves the existence of part of the result for $s=0$ in the interval $[0,T_1]$.  The uniqueness part follows from Theorem \ref{unique-1}.  This proves the existence and uniqueness of the solution in $C([0,T_1]; H^1(\mathbb{R}^n)) \cap C^1([0,T_1]; L^2(\mathbb{R}^n))$.  
 
 We can repeat the same argument on the intervals $[j T_1, (j+1)T_1]$ for $j=1,2, \ldots j_0$ with $T-T_1 < j_0 T_1 \leq T$ and on the interval $[(j_0+1)T_1, T],$ and so we have proved the existence and uniqueness of solutions in $C([0,T]; H^1(\mathbb{R}^n) )\cap C^1([0,T]; L^2(\mathbb{R}^n))$.    
 
 Next we prove \r{estimate1} for $s=0$.  Since $\kappa(|u|^2)h(t,x,u)$ is  bounded, it follows that
 \[
 \|G(t,\cdot,u)\|_{L^2(\mathbb{R}^n)}=\|\kappa(|u|^2)h(t,\cdot,u) u\|_{L^2(\mathbb{R}^n)}  \leq C_0(T,\rho) \|u\|_{L^2(\mathbb{R}^n)},
 \]
 and so it follows that
 \be{bdu-0}
 \begin{split}
 & \sup_{[0,T]} \| u(t,\cdot)\|_{H^{1}(\mathbb{R}^n)}+ \sup_{[0,T]} \|\p_tu(t,\cdot)\|_{L^2(\mathbb{R}^n)} \leq \\
 &  CT(\|\varphi\|_{H^{1}(\mathbb{R}^n)}+  \|\psi\|_{L^2(\mathbb{R}^n)} +  \int_0^T \|r(t,\cdot)\|_{L^2(\mathbb{R}^n)}  \, \d t+ C_0(T,\rho) \int_0^T \|u(t,\cdot)\|_{L^2(\mathbb{R}^n)}\, \d t).
  \end{split}
  \ee
  In particular this implies that if
  \[
  \begin{split}
  & \mathcal{E}_1(t)= \| u(t,\cdot)\|_{H^{1}(\mathbb{R}^n)}+\|\p_tu(t,\cdot)\|_{L^2(\mathbb{R}^n)}, \text{ and } \\
 & E_0=\|\varphi\|_{H^1(\mathbb{R}^n)}+  \|\psi\|_{L^2(\mathbb{R}^n)} +  \int_0^T \|r(t,\cdot)\|_{L^2(\mathbb{R}^n)}  \, \d t,
 \end{split}
  \]
   then for any $t\in [0,T]$, 
  \[
  \mathcal{E}_1(t)\leq  CTE_0+ C_0 CT \int_0^t \mathcal{E}_1(s) \, \d s,
  \]
  and so  in view of \r{GRW},
  \[
  \mathcal{E}_1(t) \leq C(T) E_0,  \;\ t\in [0,T],
  \]
  and this proves \r{estimate1} for $s=0$. 
 
 Next we show that if $ \varphi \in H^{s+1}$ and $\psi \in H^s(\mathbb{R}^n)$, the solution $u$  in fact satisfies $u  \in C([0,T_1]; H^{s+1}(\mathbb{R}^n) \cap C^1([0,T_1]; H^s(\mathbb{R}^n)$.  In the case $s=1,$  in view of \r{Eest},  we have
 \be{New-Est1}
 \begin{split}
&  \sup_{[0,T]} \| u(t,\cdot)\|_{H^{2}(\mathbb{R}^n)}+ \sup_{[0,T]} \|\p_tu(t,\cdot)\|_{H^1(\mathbb{R}^n)} \leq  \\
 & CT(\|\varphi\|_{H^{2}(\mathbb{R}^n)}+  \|\psi\|_{H^1(\mathbb{R}^n)}+ \int_0^T \|r(t,\cdot)\|_{H^1(\mathbb{R}^n)}  \, \d t+  \int_0^T \|G(t,\cdot,u)\|_{H^1(\mathbb{R}^n)} \d t),
 \end{split}
\ee
 provided the right hand side is finite.  Let us denote $\kappa(|u|^2) h(t,x,u)= G_1(t,x,u)$,
 \[
 \p_{x_j} G(t,x,u)=  \p_{x_j} (G_1(t,x,u) u)=  (\p_{x_j} G_1)(t,x,u)  u+ (\p_{u} G_1)(t,x,u) u \p_{x_j} u+ G_1(t,x,u) \p_{x_j} u
  \]
  Again using that $u (\p_uG_1)(t,x,u)$, $ (\p_{x_j} G_1)(t,x,u)$,  and $G_1(t,x,u)$ are bounded, it follows  that there exists $C_1=C_1(T,\rho)$ such that 
  \[
\int_0^T \|G(t,\cdot,u)\|_{H^1(\mathbb{R}^n)} \d t \leq C_1 \int_0^T \| u(t,\cdot)\|_{H^{1}(\mathbb{R}^n)}\,\d t.
  \]
    If we denote
   \[
 \mathcal{E}_2(t)= \| u(t,\cdot)\|_{H^{2}(\mathbb{R}^n)}+\|\p_tu(t,\cdot)\|_{H^1(\mathbb{R}^n)},
   \]
  and recall the definition of $E_2$ from \r{estimate1},   then for any $t\in [0,T]$, 
  \[
  \mathcal{E}_2(t)\leq  CT E_2  + C_1 T \int_0^t \mathcal{E}_2(s) \, \d s,
  \]
  and again from \r{GRW},
  \[
  \mathcal{E}_2(t) \leq C(T) E_2  .
  \]
  and this proves \r{estimate1} for $s=1$.

   The general case follows from the formula
  \be{diff-g}
  \partial_{x}^\alpha (G_1(t,x,u) u)= C_0(t,x,u) u+ \sum_{k=1}^{|\alpha|} \sum_{0<|\beta|\leq |\alpha|} C_\beta(t,x,u) (\partial_{x}^{\beta_1} u) (\partial_{x}^{\beta_2}u) \ldots (\partial_{x}^{\beta_k} u),
  \ee
  where $\beta$ is a collection of multi-indices, $\beta=(\beta_1, \ldots, \beta_k)$,  $\beta_j \in \mathbb{N}^{n+1}$,  $|\beta|=|\beta_1|+ \ldots |\beta_k|$,  and $C_\beta(t,x,u)$ is a function involving derivatives of $G_1(t,x,u)$.  We have shown this is true for $|\alpha|=1$, and the general case  can be proved by induction.  So we have
 \be{aux-1}
\| G_1(t,\cdot,u) u\|_{H^s(\mathbb{R}^n)} \leq C \|u\|_{L^2(\mathbb{R}^n)} + C\sum_{k=1}^{|\alpha|} \sum_{0<|\beta|\leq |\alpha|} \|(\partial_{x}^{\beta_1} u) (\partial_{x}^{\beta_2}u) \ldots (\partial_{x}^{\beta_k} u)\|_{L^2(\mathbb{R}^n)}.
\ee

   We also need the Gagliardo-Nirenberg inequality, see for example \cite{Fiorenza}:  For $|\alpha|\leq m$,
\be{GN}
\| \partial_x^\beta u\|_{L^{\frac{2m}{|\beta|}}} \leq C \|u\|_{L^\infty}^{1-\frac{|\alpha|}{m}} \left(\sum_{|\gamma|\leq m} \|\partial_x^\gamma u\|_{L^2}\right)^{\frac{|\alpha|}{m}},
\ee
and here we are using that the norms are taken over a compact subset of $\mathbb{R}^n$, determined by the support of $G_1$ in $x$. We then apply  H\"older's inequality to \r{diff-g} in the following way
\[
\begin{split}
 \|( \p_{x}^{\beta_1} u)( \p_{x}^{\beta_2} u) \ldots (\p_{x}^{\beta_k} u)\|_{L^2} & \leq \|\p_{x}^{\beta_1} u \|_{L^{p_1}}\|\p_{x}^{\beta_2} u \|_{L^{p_2}}\ldots \|\p_{x}^{\beta_k}u  \|_{L^{p_k}},  \\
 \text{ with }  p_j=\frac{2m}{|\beta_j|}, & \; m=\sum_{j=1}^k |\beta_j|, \text{ so}  \sum_{j=1}^k \frac{1}{p_j}=\frac12.
\end{split}
\]
 Then \r{GN} gives
\[
\|\p_{x}^{\beta_j} u\|_{L^{p_j}(\mathbb{R}^n)} \leq \|u\|_{L^{\infty}(\mathbb{R}^n)}^{1-\frac{|\beta_j|}{m}} 
\left(\sum_{|\gamma|\leq m} \|\partial_x^\gamma u\|_{L^{2}(\mathbb{R}^n)}\right)^{\frac{|\beta_j|}{m}},
\]
and so we conclude that if $m=\sum_{j=1}^k |\beta_j|$, then
\[
\| (\p_{x}^{\beta_1} u )(\p_{x}^{\beta_2} u )\ldots (\p_{x}^{\beta_k} u)\|_{L^2} \leq \|u\|_{L^{\infty}(\mathbb{R}^n)}
\left(\sum_{|\gamma|\leq m} \|\partial^\gamma u\|_{L^{2}(\mathbb{R}^n)}\right),
\]
and we deduce from  \r{aux-1} that for non-negative integers $s$ there exists $C_s=C_s(T,\rho,s)$ such that 
\be{HSNL}
\|G_1(t,\cdot,u) u\|_{H^s(\mathbb{R}^n)} \leq C_s \|u\|_{H^s(\mathbb{R}^n)},
\ee
and the energy estimate \r{Eest} gives

 \[
 \begin{split}
 &  \sup_{[0,T]} \| u(t,\cdot)\|_{H^{s+1}(\mathbb{R}^n)}+ \sup_{[0,T]} \|\p_tu(t,\cdot)\|_{H^s(\mathbb{R}^n)} \leq  \\
& CT\bigg(\|\varphi\|_{H^{s+1}(\mathbb{R}^n)}+  \|\psi\|_{H^s(\mathbb{R}^n)}+   \int_0^T \|r(t,\cdot)\|_{H^s(\mathbb{R}^n)}  \, \d t+ \int_0^T \|G(t,\cdot,u)\|_{H^s(\mathbb{R}^n)} \d t\bigg),
\end{split}
\]
and if we denote   
   \[
  \mathcal{E}_{s+1}(t)= \| u(t,\cdot)\|_{H^{s+1}(\mathbb{R}^n)}+\|\p_tu(t,\cdot)\|_{H^s(\mathbb{R}^n)}, \\
    \] 
  and use \r{estimate1}, it follows from \r{HSNL} for any $t\in [0,T]$, 
  \[
  \mathcal{E}_{s+1}(t)\leq  CTE_s+ C_s  T\int_0^t \mathcal{E}_{s+1}(\mu) \, \d\mu,
  \]
  and so
  \[
  \mathcal{E}_{s+1}(t) \leq C(T,C_s) E_s  .
  \]
This completes the proof of Theorem~\ref{cut-off}.  
\end{proof}

We  also need a stability estimate for solutions of \eqref{Weq2}.
\begin{theorem}\label{stability} Let $u_j(t,x)$ satisfy \eqref{Weq2} with initial data $(\varphi_j,\psi_j)$,  $j=1,2$ and right hand side $r_j(t,x)$,   
such that $\p_x \varphi\in H^s(\mathbb{R}^n)$ and $\psi  \in H^s(\mathbb{R}^n)$,  and $r_j\in L^1(\mathbb{R}; L^2(\mathbb{R}^n))$,
$s=0$ or $s=1$.  Let
\[
\mathcal{A}_s=  \|\varphi_1-\varphi_2\|_{H^{s+1}(\mathbb{R}^n)} + \|\psi_1-\psi_2\|_{H^s(\mathbb{R}^n)} +
 \int_0^T \|r_1(t,\cdot)-r_2(t,\cdot)\|_{H^s(\mathbb{R}^n)}  \, \d t.
\]
 Then for any dimension $n$, there exists  $C=C(T,\rho,h)$ such that
\be{stability1}
 \sup_{[0,T]} \|u_1(t,\cdot)-u_2(t,\cdot)\|_{H^{1}(\mathbb{R}^n)}+  \sup_{[0,T]} \|\p_t u_1(t,\cdot) -\p_t u_2(t,\cdot)\|_{L^2(\mathbb{R}^n)} \leq  C \mathcal{A}_0.
 \ee
 If $n=2,3$, and $s=1$, as in Theorem ~\ref{cut-off}, we denote
 \[
 E_{j,0}=  \|\varphi_j\|_{H^1(\mathbb{R}^n)} + \|\psi_j\|_{L^2(\mathbb{R}^n)}.
 \]
 In this case, there exists  $C=C(T,\rho, E_{1,0}, E_{2,0})$ such that
 \be{stability2}
 \sup_{[0,T]} \|u_1(t,\cdot)-u_2(t,\cdot)\|_{H^{2}(\mathbb{R}^n)}+  \sup_{[0,T]} \|\p_t u_1(t,\cdot) -\p_t u_2(t,\cdot)\|_{H^1(\mathbb{R}^n)}  \leq C \mathcal{A}_1.
\ee
\end{theorem}
\begin{proof}  In the case $s=0$,  let
\[
\mathcal{B}_0(t)= \|u_1(t,\cdot)- u_2(t,\cdot)\|_{H^{1}(\mathbb{R}^n)} + \|\p_tu_1(t,\cdot)-\p_t u_2(t,\cdot)\|_{L^2(\mathbb{R}^n)}.
\]
We deduce from \r{Eest} that
\[
\begin{split}
\mathcal{B}_0(t)  &\leq 
 CT\|\varphi_1-\varphi_2\|_{H^1(\mathbb{R}^n)} + CT\|\psi_1-\psi_2\|_{L^{2}(\mathbb{R}^n)}+ CT\int_0^T \|r_1(t,\cdot)-r_2(t,\cdot)\|_{L^2(\mathbb{R}^n)}  \, \d t  \\
& \quad+CT\int_0^T \|G(t,\cdot,u_1)- G(t,\cdot,u_2)\|_{L^2(\mathbb{R}^n)}\, \d t.
\end{split}
\]
But since $G(t,x,u_1)-G(t,x,u_2)= \mathcal{G}(t,x,u_1,u_2) (u_1-u_2)$ and $\mathcal{G}$ is bounded, it  follows that
\[
\mathcal{B}_0(t)\leq  CT\mathcal{A}_0 + C T \int_0^t \mathcal{B}_0(r)\, \d r,
\]
and therefore follows from Gronwall's inequality  \r{GRW} that
\[
\mathcal{B}_0(t)\leq C(T) \mathcal{A}_0 ,
\]
which proves the first  inequality  in \r{stability1}.

In the case $s=1$,  again energy estimates \r{Eest} give
\[
\begin{split}
 \mathcal{B}_1(T)&=  \|u_1(T,\cdot)-u_2(T,\cdot)\|_{H^{2}(\mathbb{R}^n)}+ \|\p_tu_1(T,\cdot)-\p_t u_2(T,\cdot)\|_{H^{1}(\mathbb{R}^n)}  \\
& \leq CT\mathcal{A}_1+ CT \int_0^T \|G(t,\cdot,u_1)- G(t,\cdot,u_2)\|_{H^1(\mathbb{R}^n)}\, \d t.
\end{split}
\]
But since $G(t,x,u_1)-G(t,x,u_2)= \mathcal{G}(t,x,u_1,u_2)(u_1-u_2)$,  if follows that
\[
\begin{split}
& \p_{x_j} (G(t,x,u_1)-G(t,x,u_2))= (\p_{x_j} \mathcal{G})(t,x,u_1,u_2)(u_1-u_2)+ \mathcal{G}(t,x,u_1,u_2)\p_{x_j}(u_1-u_2) \\
&\quad+(\p_{u_1} \mathcal{G})(t,x,u_1,u_2)(\p_{x_j} u_1) (u_1-u_2)+ (\p_{u_2} \mathcal{G})(t,x,u_1,u_2)(\p_{x_j} u_2) (u_1-u_2).
\end{split}
\]
Therefore
\[
\begin{split}
 & \|G(t,\cdot,u_1)- G(t,\cdot,u_2)\|_{H^1(\mathbb{R}^n)}\leq C \|u_1-u_2\|_{H^1(\mathbb{R}^n)} \\
 &\quad+ C \|u_1\|_{H^1(\mathbb{R}^n)}\|u_1-u_2\|_{L^\infty(\mathbb{R}^n)}+ C \|u_2\|_{H^1(\mathbb{R}^n)}\|u_1-u_2\|_{L^\infty(\mathbb{R}^n)}.
 \end{split}
 \]    
 We know from Theorem \ref{cut-off} that 
 \[
 \|u_j\|_{H^1(\mathbb{R}^n)}\leq C_j E_{0,j}, \;\ j=1,2.
 \]
 Since $n\leq 3$, the Sobolev Embedding Theorem gives that  for $t$ fixed, 
 \[
 \|u_1(t,\cdot )-u_2(t,\cdot)\|_{L^\infty(\mathbb{R}^n)}\leq  C(n) \|u_1(t,\cdot)-u_2(t,\cdot)\|_{H^2(\mathbb{R}^n)},
 \]
 and therefore we find that there exists $C=C(T, \rho, E_{1,0}, E_{2,0})$, such that
\[
  \|G(t,\cdot,u_1)- G(t,\cdot,u_2)\|_{H^1(\mathbb{R}^n)}\leq  C  \|u_1-u_2\|_{H^2(\mathbb{R}^n)}.
 \]  
and so for any $t\in [0,T],$
\[
\mathcal{B}_1(T)\leq CT  \mathcal{A}_1+ CT\int_0^T \mathcal{B}_1(t)\, \d t,
\]
and it follows from \r{GRW} that 
\[
\mathcal{B}_1(t)\leq C(T)\mathcal{A}_1 .
\]
This proves the theorem.
\end{proof}

 Now we discuss properties of solutions of the more general case, 
\be{Weq1}
\begin{split}
 u_{tt}-\Delta_g u &+ f(t,x,u)=0, \\
u(0,x)= \varphi(x) &,  \;\ u_t(0,x)=\psi(x)
\end{split}
\ee
 established in several degrees of generality in \cite{Grillakis,Kap1,Kap2,Ginibre-Velo,ShStruwe1,ShStruwe2}.  We assume that $f: [0,T] \times \mathbb{R}^n \times \mathbb{C} \longmapsto \mathbb{C}$,  is continuous and $f(t,x,0)=0$.   In general, as shown in \cite{Keller},  such equations have solutions that blow-up at a finite time, so to obtain existence, uniqueness and regularity of solutions of \r{Weq1}, we need to make additional assumptions about the behavior of $f(t,x,u)$ for $u\sim 0$ and for $u\sim \infty$, uniformly on $x$ and $t$.  We follow the work of Kapitanskii \cite{Kap2}, and we pick $\kappa \in C^\infty(\mathbb{R})$ such that  $\kappa(s)=1$ if $|s|>2$ and $\kappa(s)=0$ if $|s|<1$ and define
\[
F_0(t,x,u)= (1-\kappa(|u|^2)) f(t,x,u) \text{ and }  F_1(t,x,u)= \kappa(|u|^2) f(t,x,u).
\]

We assume that for $T>0$, 
\begin{enumerate}[(H1)]
\item $\int_{[0,T]\times \mathbb{R}^n} |F_0(t,x,u)|\, \d x \,\d t= I(u)<\infty$, 
\item  $|F_0(t,x,u)- F_0(t,x,v)| \leq C |u-v|$, for all   $t\in [0,T]$, and for all $x\in \mathbb{R}^n$,  \\
\item For any function $u \in L^{\infty}([0,T]; H^{s+1}(\mathbb{R}^n))$, $s>0$,  $F_0(t,s,u)\in L^1([0,T]; H^s(\mathbb{R}^n))$, and for
$u_1, u_2\in L^{\infty}([0,T]; H^{s+1}(\mathbb{R}^n))$, we have
\[
\int_0^T \|F_0(t,\cdot,u_1(t,\cdot))- F_0(t,\cdot,u_2(t,\cdot))\|_{H^s(\mathbb{R}^n)}\, \d t \leq C(T)\Big( 1+ \sup_{[0,T]} \|u_1-u_2\|_{H^{s+1}(\mathbb{R}^n)}\Big)
\]
with $C(T) \rightarrow 0$ when $T\rightarrow 0$.  Since $H^s(\mathbb{R}^n)\cap  L^\infty(\mathbb{R}^n)$ is closed under the composition with $C^\infty$ functions (see for example \cite{KaPo}), this holds if for example, $F_0(t,x,u)$ is $C^\infty$.
\end{enumerate}

We also need to make some assumptions on the growth of $f(t,x,u)$ for $u\sim \infty$.  We assume there exists
 $p\in [1, \frac{n+2}{n-2}]$ 
 such that for all $t\in [0,T]$ and  $x\in \mathbb{R}^n$, we have
\begin{enumerate}[(F1)]
\item $|F_1(t,x,z)|\leq C |z|^p$, \\
\item $|F_1(t,x,z_1)-F_1(t,x,z_2)| \leq C ( |z_1|^{p-1}+ |z_2|^{p-1}) |z_1-z_2|$, \\
\item we have
\[
\begin{split}
|(\partial_z F_1)(t,x,z_1)-(\partial_z F_1)(t,x,z_2)| & \leq C ( |z_1|^{p-1}+ |z_2|^{p-1}) |z_1-z_2|, \text{ and }\\
|(\partial_{\overline{z}}F_1)(t,x,z_1)-(\partial_{\overline{z}}F_1)(t,x,z_2)| & \leq C ( |z_1|^{p-1}+ |z_2|^{p-1}) |z_1-z_2|,
\end{split}
\] 
\item $|(\partial_{x_j}F_1)(t,x,z_1)-(\partial_{x_j}F_1)(t,x,z_2)| \leq C ( |z_1|^{p-1}+ |z_2|^{p-1}) |z_1-z_2|$, \\
\item there exists  $\delta>0$ such that,  provided $|x-y|<\delta$,
\[ 
\begin{split}|F_1(t,x,z)-F_1(t,y,z)| & \leq C(\delta) |z|^p |x-y| \text{ and } \\
 |(\partial_xF_1)(t,x,z)-(\partial_xF_1)(t,y,z)| & \leq C(\delta) |z|^p |x-y|,
 \end{split}
\]
\item there exists $\mathcal{H}(t,x,z): [0,T] \times \mathbb{R}^n \times \mathbb{C}$,  such that $\mathcal{H}(t,x,z)\geq 0$, and  such that
\[ 
F_1(t,x,z)= \frac{\partial}{\partial \overline{z}} \mathcal{H}(t,x,z);
\]
and there exists a constant $C$ such that for all $t\in [0,T]$ and  $x\in \mathbb{R}^n$, we have
\[
\partial_t \mathcal{H}(t,x,y) \leq C \mathcal{H}(t,x,y).
\]
\end{enumerate}

We can think of $\mathcal{H}$ as an energy which a priori can grow no faster than exponentially. 
For example, in the case discussed above, $f(t,x,u)=  \g(x, |u|^2) u$, with $\g$ and compactly supported in $x$.  It is clear that $F_0(t,x,y)$ satisfies the assumptions above, and that if  $\partial_r G(x,r)= \kappa(r^2) r\g(x,r^2)$, $G(x,0)=0$, then
\[ 
\mathcal{H}(t,x,z)= G(x,|z|) \geq 0 \text{ and }  \partial_t \mathcal{H}(t,x,z)=0 \leq \mathcal{H}(t,x,z).
\]
So we need to assume that $G\geq 0$.

We recall the definition of Besov spaces $\mathcal{B}_{\rho,q}^r$, $p,q>0$ and $r\in \mathbb{R}$, and we adopt the convention 
$\mathcal{B}_\rho^r:= \mathcal{B}_{\rho,2}^r$.  Let $\psi \in C_0^\infty(\mathbb{R}^n)$, $\psi(\xi)=1$ if $|\xi|<\frac12$ and $\psi(\xi)=0$ if $|\xi|>1$.  Let $\mathcal{F}$ denote the Fourier transform and for $f\in \mathcal{S}'(\mathbb{R}^n)$, define the operators $S_k(f)$ and $\Delta_k(f)$  as
\[ 
\mathcal{F}(S_k(f))(\xi)= \psi(2^{-k}\xi) \mathcal{F}(f)(\xi), \;\ \Delta_k(f)= S_{k+1} f- S_k (f).
\]
Then
\[
f= S_0(f) + \sum_{k=0}^\infty \Delta_k(f),
\]
and we say that $f\in \mathcal{B}_{\rho,q}^r$, if
\[
\|S_0(f)\|_{L^\rho(\mathbb{R}^n)} + \bigg[\sum_{k=0}^\infty (2^{rk} \|\Delta_k (f)\|_{L^\rho(\mathbb{R}^n)})^q\bigg]^{\frac{1}{q}}<\infty.
\]

\begin{theorem}[\cite{Kap1,Kap2}] \label{thm_K} 
Suppose that  $F_0(t,x,y)$ and $F_1(t,x,y)$ satisfy the hypotheses above and that the initial data $(\varphi,\psi)$ satisfies $\nabla_x \varphi\in L^2(\mathbb{R}^n)$ and $\psi \in L^2(\mathbb{R}^n)$. Then  the nonlinear equation  \r{Weq1} has unique weak solution $u(t,x)$ which satisfies $\nabla_x u, u_t \in C([0,T]; L^2(\mathbb{R}^n))$ and  
$u, u_t \in L^q([0,T], \mathcal{B}_\rho^{r}(\mathbb{R}^n))$ with $r\in (\frac{n-3}{2(n-1)}, 1]$ and 
\[
\frac{1}{\rho}=\frac12 -\frac{2(1-r)}{n+1}, \;\  \frac{1}{q}=(1-r)\frac{n-1}{n+1}.
\]
   Moreover, if  $s<p$ for $p<2$, or if $s\in (0,2]$, for $p\geq 2$,  and $\varphi\in H^{s+1}(\mathbb{R}^n)$ and $\psi \in H^s(\mathbb{R}^n)$, then in fact 
$u, u_t \in C([0,T]; H^s(\mathbb{R}^n))$ and  
$u, u_t \in L^q([0,T], \mathcal{B}_\rho^{r+s}(\mathbb{R}^n))$.
\end{theorem}

In this generality, this result is due to Kapitanskii (Theorem 0.10 of \cite{Kap2}), but the case where $g$ is the Euclidean metric and $f(t,x,u)= |u|^{p-1} u$ is due to Shatah and Struwe \cite{ShStruwe1,ShStruwe2} and Grillakis \cite{Grillakis}.



\end{document}